\pdfoutput=1
\documentclass[12pt]{amsart}
\usepackage{float}
\usepackage{xspace}

\usepackage{amsmath,amsfonts,amssymb,mathabx,latexsym,mathtools,changepage}
\usepackage[shortlabels]{enumitem}
\usepackage[usenames, dvipsnames]{xcolor}
\usepackage{hyperref}
\usepackage{ytableau}
\usepackage{centernot}
\usepackage{tikz-cd}
\usepackage{tikz}
\usetikzlibrary{decorations.markings}
\usepackage{stmaryrd}
\usepackage{wrapfig}
 \usepackage{url}
\usetikzlibrary{calc, shapes, backgrounds,arrows,positioning,plotmarks}
\tikzset{>=stealth',
  head/.style = {fill = white, text=black},
  plaque/.style = {draw, rectangle, minimum size = 10mm}, 
  pil/.style={->,thick},
  junct/.style = {draw,circle,inner sep=0.5pt,outer sep=0pt, fill=black}
  }
\definecolor{lightergray}{gray}{0.85}
\definecolor{lightestgray}{gray}{0.93}
 
\newtheorem{theorem}{Theorem}[section]
\newtheorem{lemma}[theorem]{Lemma}

\newtheorem{proposition}[theorem]{Proposition}
\newtheorem{corollary}[theorem]{Corollary}
\newtheorem*{claim*}{Claim}
\theoremstyle{definition}

\newtheorem{construction}[theorem]{Construction}

\newtheorem{examplex}[theorem]{Example}
\newenvironment{example}
  {\pushQED{\qed}\examplex}
  {\popQED\endexamplex}

\theoremstyle{remark}
\newtheorem{remark}[theorem]{Remark}

\numberwithin{equation}{section}

\usepackage[colorinlistoftodos]{todonotes}

  \newcommand{\Plus}{\mathord{\begin{tikzpicture}[baseline=0ex, line width=1.2, scale=0.11]
\draw (1,0) -- (1,2);
\draw (0,1) -- (2,1);
\end{tikzpicture}}}

\newcommand\mydef[1]{{\bf #1}}

\newcommand \mdCR[1]{{\tt md}(#1)}
\newcommand \avoid[1]{{\sf avoid}_{321}(#1)}
\newcommand \perm{{\sf perm}}

\newcommand \dTab{{\sf dTab}}

\newcommand{\maxexcited}{D_{\tt top}}
\newcommand{\zip}{D_{\tt zip}}
\newcommand{\Kzip}{D_{\tt zip}^{K}}

\newcommand{\SEelb}{
\raisebox{-0.2em}{  \begin{tikzpicture}[scale=0.25]
\begin{scope}[scale=0.6,thick]
\draw[line width = .25ex, blue] (1.5,5.5) -- (1.5,7)--(1.5,5.5)-- (3,5.5);
\draw (0,4) -- (0,7)--(3,7)--(3,4)--(0,4);
\end{scope}
\end{tikzpicture}}
}

\newcommand{\blank}{
 \raisebox{-0.2em}{  \begin{tikzpicture}[scale=0.25]
\begin{scope}[scale=0.6,thick]
\draw 
(0,4) -- (0,7)--(3,7)--(3,4)--(0,4);
\end{scope}
\end{tikzpicture}}
}

\newcommand{\GL}{{\rm GL}}

\newcommand{\dem}{\delta}

\newcommand{\Pipes}{{\sf Pipes}}

\newcommand{\KPipes}{\overline{\sf Pipes}}

\newcommand{\skewexcited}{{\sf SEYD}}
\newcommand{\nilp}{{\sf NILP}}
\newcommand{\elb}{{\sf elbows}}
\newcommand{\blk}{{\sf blanks}}
\newcommand\WB[1]{{{#1}_{W}}}
 \newcommand\SB[1]{{{#1}_{S}}}
  \newcommand\Tab[1]{{{\sf T}}_{#1}}
   \newcommand\fTab[1]{{{\sf F}}_{#1}}
   
   \newcommand\CM[1]{{\tt Zip}{#1}}

\newcommand{\kskewexcited}{\overline{\sf SEYD}}

\newcommand{\wt}{{\tt wt}}

\newcommand\gap{{\tt gap}}
\newcommand\codePD[1]{{\mathcal{R}_{#1}}}
\newcommand{\Sym}{S}
 \newcommand\Pos[1]{{#1}_{+}}
 \newcommand{\reg}{{\rm reg}}
\newcommand{\UD}{{lower decomposable}}

\begin{document}

\title[Regularity of $321$-avoiding Kazhdan--Lusztig varieties]{Castelnuovo--Mumford regularity for $321$-avoiding Kazhdan--Lusztig varieties}

\author{Colleen Robichaux}
\address[CR]{
Dept.~of Mathematics,
University of California, Davis,
Davis, CA  \newline \indent 95616, USA}
\email{robichaux@ucdavis.edu}
\thanks{Colleen Robichaux was supported by the NSF GRFP No. DGE 1746047, NSF RTG No. DMS 1937241, and NSF MSPRF No. DMS 2302279.}

\date{\today}

\begin{abstract}
 We prove that the Castelnuovo--Mumford regularity of 321-avoiding Kazhdan--Lusztig varieties can be computed combinatorially in terms of $K$-skew excited Young diagrams. We present an algorithm which gives a lower bound for this regularity and then describe a setting in which this algorithm in fact computes the regularity. This algorithm specializes to compute the regularity of all two-sided mixed ladder determinantal varieties.
\end{abstract}

\maketitle

\section{Introduction}
\emph{Kazhdan--Lusztig varieties} are generalized determinantal varieties introduced by Woo--Yong \cite{WooYongSings} to study singularities of \emph{Schubert varieties}. \emph{Matrix Schubert varieties}, introduced by Fulton \cite{Fulton.Flags}, and \emph{ladder determinantal varieties}, introduced by Abhyankar \cite{Abhyankar}, are well-studied families of Kazhdan--Lusztig varieties, see \cite{Conca,HT,KM}. Kazhdan--Lusztig varieties indexed by {$321$-avoiding} permutations form a large class of Kazhdan--Lusztig varieties with homogeneous defining ideals. As proven by Escobar--Fink--Rajchgot--Woo \cite{EFRW}, these include \emph{two-sided mixed ladder determinantal varieties} as special cases.

The \emph{Castelnuovo--Mumford regularity} of a graded module is an invariant used to measure its complexity. In general, regularity may be computed using the minimal free resolution of the module in terms of its Betti numbers. 
Since Kazhdan--Lusztig varieties are Cohen--Macaulay, one may instead compute their regularities in terms of degrees of \emph{unspecialized Grothendieck polynomials}, introduced by Woo--Yong \cite{WooYongGrobner}.

We leverage this fact to give a combinatorial rule for the regularity of $321$-avoiding Kazhdan--Lusztig varieties. Using this rule, we define a combinatorial algorithm that provides a lower bound for this regularity. In particular cases, which include all two-sided mixed ladder determinantal varieties, this algorithm precisely computes the corresponding regularities. 
This paper generalizes previous work of Rajchgot, Weigandt, and the author \cite{RRW} which gives a tableaux-based formula to compute the regularity of Kazhdan--Lusztig varieties indexed by \emph{Grassmannian} permutations. Our specialized results for two-sided ladders extend work of Ghorpade--Krattenthaler~\cite{GK15} which computes the \emph{$a$-invariants} for particular two-sided mixed ladder determinantal varieties in terms of lattice paths.

\subsection{Summary of Results}
We introduce \emph{$K$-skew excited Young diagrams} to compute regularities of $321$-avoiding Kazhdan--Lusztig varieties:

\begin{theorem}\label{thm:321321IndReg} Suppose $v\geq w$ are $321$-avoiding permutations.  
Then
\[ \reg (\mathbb{C}[{\bf z}^v]/J_{v,w})=\max\{|D| \, : \, D\in\kskewexcited{(v,w)}\}-{\sf inv}(w)\geq|\Kzip(v,w)|-{\sf inv}(w).
\]
When $v,w$ is a {\UD} pair, equality is achieved.
\end{theorem}

Here ${\sf inv}(w)$ denotes the inversion number of the permutation $w$. In Section~\ref{sec:skewEYD}, we define $\kskewexcited{(v,w)}$, the set of $K$-skew excited Young diagrams for $v$ and $w$. 
In Section~\ref{sec:degProof}, we construct the distinguished diagram $\Kzip(v,w)\in \kskewexcited{(v,w)}$. We define \emph{\UD} in Section~\ref{sec:321321alg}. Theorem~\ref{thm:321321IndReg} is proven in Section~\ref{sec:321321reg}.

\begin{example}\label{ex:introThm} Below is 
$\Kzip(v,w)$ for some $v,w\in S_{14}$.\footnote{Here $v=(4,11,12,13,14,1,2,3,5,6,7,8,9,10)$ and $w=(1,4,7,2,8,3,11,5,6,12,9,13,14,10)$.}
  \[
  \ytableausetup
{boxsize=0.8em}
{\begin{ytableau}
 \ &  + &  +  \\
  \textcolor{blue}{\Plus} &  \ &  + &  + &+ &  \ &  \ &  +  &  +  \\
  + &  \textcolor{blue}{\Plus} &  \ &  + &+ &  \ &  \ &  + &  +   \\
 \ &  + &  \textcolor{blue}{\Plus} &  \ &+ &  \ &  \textcolor{blue}{\Plus} &  \  &  +  \\
 \textcolor{blue}{\Plus} &  \ &  + &  \textcolor{blue}{\Plus} &\ &  \textcolor{blue}{\Plus} &  \ &  \  &  +  \\
\end{ytableau}}
\]
In this case, the pair $v,w$ is {\UD}. Thus
Theorem~\ref{thm:321321IndReg} determines  
\[\reg (\mathbb{C}[{\bf z}^v]/J_{v,w})=|\Kzip(v,w)|-{\sf inv}(w)=24-17=7.\]
By construction, 
$\reg (\mathbb{C}[{\bf z}^v]/J_{v,w})$ counts 
the number of bold blue plusses above.
\end{example}
As with regularity, the \emph{a-invariant} of a module is an invariant providing data that may increase efficiency in computations, see \cite{BH92} for  discussion. Using Theorem~\ref{thm:321321IndReg}, we compute the $a$-invariant for $321$-avoiding Kazhdan--Lusztig varieties:
\begin{corollary}\label{cor:321321IndAInv} 
Suppose $v\geq w$  are $321$-avoiding permutations.  
Then
\[ a(\mathbb{C}[{\bf z}^v]/J_{v,w})=\max\{|D| \, : \, D\in\kskewexcited{(v,w)}\}-{\sf inv}(v)\geq|\Kzip(v,w)|-{\sf inv}(v).
\]
When $v,w$ is a {\UD} pair, equality is achieved.
\end{corollary}
Corollary~\ref{cor:321321IndAInv} is proven in Section~\ref{sec:321321reg}. 

\begin{example}
    Taking $v,w$ as in Example~\ref{ex:introThm}, Corollary~\ref{cor:321321IndAInv} computes
    \[a(\mathbb{C}[{\bf z}^v]/J_{v,w})=|\Kzip(v,w)|-{\sf inv}(v)=24-39=-15.\]
    By construction,  $|a(\mathbb{C}[{\bf z}^v]/J_{v,w})|$ counts the empty cells in the illustration  of $\Kzip(v,w)$.
\end{example}

We then prove  specializations of Theorem~\ref{thm:321321IndReg} and Corollary~\ref{cor:321321IndAInv} for two-sided mixed ladder determinantal varieties. These are all $321$-avoiding  Kazhdan--Lusztig varieties for which $v,w$ are {\UD}, so our algorithm computes their regularities.

\subsection{Context in Literature}
The regularity of matrix Schubert varieties has recently become well-understood. Initial work of Ren, Rajchgot, St.~Dizier, Weigandt, and the author \cite{RRRSW} gives a combinatorial formula for the regularity of Grassmannian matrix Schubert varieties in terms of integer partitions. 
The work of Pechenik--Speyer--Weigandt \cite{PSW} uses poset-theoretic techniques to compute the regularity of arbitrary matrix Schubert varieties in terms of permutation statistics. 

Pan--Yu \cite{PanYu} use \cite{PSW} to give a diagrammatic regularity formula for matrix Schubert varieties.  The results in \cite{PSW} have been re-proven by Dreyer--Mészáros--St.~Dizier \cite{DMS} using saturated chains in Bruhat order. Formulas for regularities of matrix Schubert varieties for particular cases \cite{Hafner,RRW} as well as tangent cones of Schubert varieties \cite{YongCM} have also been studied.

The results of Rajchgot, Weigandt, and the author \cite{RRW} give a combinatorial formula to compute the regularity of Grassmannian Kazhdan--Lusztig varieties. Due to a correspondence
with matrix Schubert varieties, these results in \cite{RRW} may be recovered using \cite{PSW}. In general, $321$-avoiding Kazhdan--Lusztig varieties are not isomorphic to matrix Schubert varieties. Thus Theorem~\ref{thm:321321IndReg} generalizes \cite{RRW} in a different direction than \cite{PSW}.

Combinatorial formulas to compute the $a$-invariant of families of \emph{one-sided ladder determinantal varieties} have been determined \cite{BH92,ConcaAinv,GK15}. Additionally, Ghorpade--Krattenthaler \cite{GK15} give an algorithm to compute the $a$-invariant for a large family of two-sided ladder determinantal varieties in terms of lattice paths. Forthcoming work of Escobar--Fink--Rajchgot--Woo \cite{EFRW} shows two-sided ladder determinantal varieties are $321$-avoiding Kazhdan--Lusztig varieties. Using this fact, we apply Theorem~\ref{thm:321321IndReg} and Corollary~\ref{cor:321321IndAInv} to compute the regularity and  $a$-invariant for two-sided ladder determinantal varieties in terms of lattice paths. We relate our formulas to \cite{GK15} by generalizing a bijection of Kreiman \cite{KreimanEYD} between lattice paths and excited Young diagrams.

\subsection{Outline}
 In Section~\ref{sec:skewexcited} we establish the combinatorial background and introduce $K$-skew excited Young diagrams. We give the geometric and commutative algebraic background in Section~\ref{sec:geomBack}. In Section~\ref{sec:mainThmPf} we define {\UD}, define $\Kzip(v,w)$, and  prove our main results of Theorem~\ref{thm:321321IndReg} and Corollary~\ref{cor:321321IndAInv}. In Section~\ref{sec:regLadd} we specialize these results to two-sided mixed ladder determinantal varieties and translate our results in terms of lattice paths.

\section{Combinatorial Background}
\label{sec:skewexcited}
For $n\in\mathbb{Z}$, let $[n]:=\{i\in\mathbb{Z}_{>0}  \, : \, i\leq n\}$.

\subsection{Pipe complexes}
\label{sec:pipeCompl}
Let $S_n$ denote the symmetric group on $n$ letters. We write $u\in S_n$ in one-line notation and let $u_i:=u(i)$ for $i\in [n]$. The \mydef{rank function} of $u\in S_n$ is defined as \[{\sf rank}_u(i,j):=\big|\{(k,u_{k})  \, : \,  k\in[i],u_{k} \in[j] \}\big|\]
for $(i,j)\in[n]^2$.
The \mydef{Rothe diagram} of $u\in \Sym_n$ is the set
	\[ D(u):= \{(i, j)\in [n]\times [n] \, :  \,  u_i > j \text{ and } u^{-1}_j > i\}.\]
	We illustrate $D(u)$ as the blank cells in the $n\times n$ grid after placing points in cells $(i,u_i)$ and drawing a line through cells which appear weakly below or weakly to the right of $(i,u_i)$ for each $i\in[n]$. Let ${\sf inv}(u):=| D(u)|$ denote the {\bf inversion number} of $u$. 
The \mydef{Lehmer code} of $u$ is the tuple ${\sf code}(u)=(c_1,\ldots,c_n)$ where $c_i$ counts the number of boxes in row $i$ of $D(u)$. Further, ${\sf code}(u)$ uniquely  encodes $u$ \cite[Proposition~2.1.2]{Manivel}.

\begin{example}\label{ex:Rothe} Below are $D(v)$ and $D(w)$ for $v=46128935(10)7$ and $w=412368597(10)$.
    
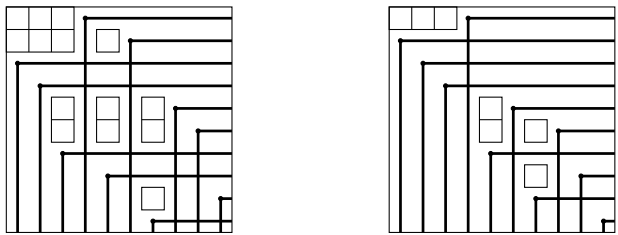
\[
\begin{tikzpicture}[scale=.3]
\draw (0,0) rectangle (10,10);

 \draw (0,10) rectangle (3,8);
\draw (0,9)--(3,9);
\draw (1,10)--(1,8);
\draw (2,10)--(2,8);

\draw (4,8) rectangle (5,9);

\draw (2,6) rectangle (3,4);
 \draw (2,5)--(3,5);
 
\draw (4,6) rectangle (5,4);
 \draw (4,5)--(5,5);

\draw (6,6) rectangle (7,4);
 \draw (6,5)--(7,5);

\draw (6,1) rectangle (7,2);

\filldraw (3.5,9.5) circle (.5ex);
\draw[line width = .2ex] (3.5,0) -- (3.5,9.5) -- (10,9.5);
\filldraw (5.5,8.5) circle (.5ex);
\draw[line width = .2ex] (5.5,0) -- (5.5,8.5) -- (10,8.5);
\filldraw (0.5,7.5) circle (.5ex);
\draw[line width = .2ex] (0.5,0) -- (0.5,7.5) -- (10,7.5);
\filldraw (1.5,6.5) circle (.5ex);
\draw[line width = .2ex] (1.5,0) -- (1.5,6.5) -- (10,6.5);
\filldraw (7.5,5.5) circle (.5ex);
\draw[line width = .2ex] (7.5,0) -- (7.5,5.5) -- (10,5.5);
\filldraw (8.5,4.5) circle (.5ex);
\draw[line width = .2ex] (8.5,0) -- (8.5,4.5) -- (10,4.5);
\filldraw (2.5,3.5) circle (.5ex);
\draw[line width = .2ex] (2.5,0) -- (2.5,3.5) -- (10,3.5);
\filldraw (4.5,2.5) circle (.5ex);
\draw[line width = .2ex] (4.5,0) -- (4.5,2.5) -- (10,2.5);
\filldraw (9.5,1.5) circle (.5ex);
\draw[line width = .2ex] (9.5,0) -- (9.5,1.5) -- (10,1.5);
\filldraw (6.5,0.5) circle (.5ex);
\draw[line width = .2ex] (6.5,0) -- (6.5,0.5) -- (10,0.5);
\end{tikzpicture}
\hspace{5em} 
\begin{tikzpicture}[scale=.3]
\draw (0,0) rectangle (10,10);

 \draw (0,10) rectangle (3,9);
 \draw (1,10) -- (1,9);
 \draw (2,10) -- (2,9);
 
\draw (4,6) rectangle (5,4);
 \draw (4,5)--(5,5);

\draw (6,5) rectangle (7,4);

\draw (6,3) rectangle (7,2);

\filldraw (3.5,9.5) circle (.5ex);
\draw[line width = .2ex] (3.5,0) -- (3.5,9.5) -- (10,9.5);
\filldraw (0.5,8.5) circle (.5ex);
\draw[line width = .2ex] (0.5,0) -- (0.5,8.5) -- (10,8.5);
\filldraw (1.5,7.5) circle (.5ex);
\draw[line width = .2ex] (1.5,0) -- (1.5,7.5) -- (10,7.5);
\filldraw (2.5,6.5) circle (.5ex);
\draw[line width = .2ex] (2.5,0) -- (2.5,6.5) -- (10,6.5);
\filldraw (5.5,5.5) circle (.5ex);
\draw[line width = .2ex] (5.5,0) -- (5.5,5.5) -- (10,5.5);
\filldraw (7.5,4.5) circle (.5ex);
\draw[line width = .2ex] (7.5,0) -- (7.5,4.5) -- (10,4.5);
\filldraw (4.5,3.5) circle (.5ex);
\draw[line width = .2ex] (4.5,0) -- (4.5,3.5) -- (10,3.5);
\filldraw (8.5,2.5) circle (.5ex);
\draw[line width = .2ex] (8.5,0) -- (8.5,2.5) -- (10,2.5);
\filldraw (6.5,1.5) circle (.5ex);
\draw[line width = .2ex] (6.5,0) -- (6.5,1.5) -- (10,1.5);
\filldraw (9.5,0.5) circle (.5ex);
\draw[line width = .2ex] (9.5,0) -- (9.5,0.5) -- (10,0.5);
\end{tikzpicture}
\]
Here ${\sf inv}(w)=7<{\sf inv}(v)=14$ and ${\sf code}(v)=(3,4,0,0,3,3,0,0,1,0)$. 
\end{example}

Define the algebra over $\mathbb Z$ generated by $\{e_u \, : \, u\in S_n\}$ with multiplication such that
\[e_ue_{s_i}=
\begin{cases}
e_{us_i} &\text{if } {\sf inv}(us_i)>{\sf inv}(u), \text{ and} \\
e_u & \text{otherwise.}
\end{cases}\]
Here $s_i$ is the simple transposition which permutes $i$ and $i+1$.

For $u\in S_n$
label the boxes of $D(u)$ along rows so that the $k$th  leftmost box in row $i$
 is assigned the label $i+k-1$.
 Given $P\subseteq D(u)$, let ${\sf word}(P)$ in $D(u)$ be the sequence formed by reading the labels of $P$ in this labeling of $D(u)$, scanning right to left across rows, from top to bottom. The \mydef{Demazure product} $\dem(P)$ of $P$ is the permutation that satisfies \[e_{s_{i_1}}\cdots e_{s_{i_k}}=e_{\dem(P)},\]
 where ${\sf word}(P)=(i_1,i_2,\ldots,i_k)$ in $D(u)$.

 Take $v,w\in S_n$ where $v\geq w$, \emph{i.e.}, $v$ is above $w$ in Bruhat order.
 Define 
 \begin{align*} 
     \KPipes(v,w)&:=\{P\subseteq D(v)\, :  \,\dem(P)=w\}, \text{ and}\\
     \Pipes(v,w)&:=\{P\in  \KPipes(v,w) \, :  \,|P|={\sf inv}(w)\}.
\end{align*}
 We illustrate $P\subseteq D(v)$ by filling each $(i,j)\in P$ with a $+$ in $D(v)$.

\begin{example}\label{ex:pipeSubword} The left two diagrams are labeled diagrams $D(v)$ and $D(w)$ for $v,w$ as in Example~\ref{ex:Rothe}. This gives ${\sf word}(D(w))=(3,2,1,5,7,6,8)$ in $D(w)$. 
The third diagram is $P\in \Pipes(v,w)$, and the fourth is some $P'\in \KPipes(v,w)$.
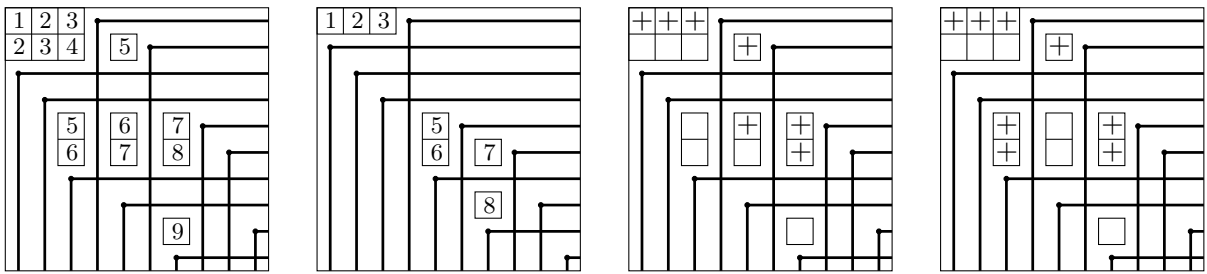
    \[
\begin{tikzpicture}[scale=.35]
\draw (0,0) rectangle (10,10);

 \draw (0,10) rectangle (3,8);
\draw (0,9)--(3,9);
\draw (1,10)--(1,8);
\draw (2,10)--(2,8);

\draw (4,8) rectangle (5,9);

\draw (2,6) rectangle (3,4);
 \draw (2,5)--(3,5);
 
\draw (4,6) rectangle (5,4);
 \draw (4,5)--(5,5);

\draw (6,6) rectangle (7,4);
 \draw (6,5)--(7,5);

\draw (6,1) rectangle (7,2);

\put(3,92){\scriptsize{$1$}}
\put(13,92){\scriptsize{$2$}}
\put(23,92){\scriptsize{$3$}}
\put(3,82){\scriptsize{$2$}}
\put(13,82){\scriptsize{$3$}}
\put(23,82){\scriptsize{$4$}}

\put(43,82){\scriptsize{$5$}}

\put(23,52){\scriptsize{$5$}}
\put(43,52){\scriptsize{$6$}}
\put(63,52){\scriptsize{$7$}}
\put(23,42){\scriptsize{$6$}}
\put(43,42){\scriptsize{$7$}}
\put(63,42){\scriptsize{$8$}}

\put(63,12){\scriptsize{$9$}}

\filldraw (3.5,9.5) circle (.5ex);
\draw[line width = .2ex] (3.5,0) -- (3.5,9.5) -- (10,9.5);
\filldraw (5.5,8.5) circle (.5ex);
\draw[line width = .2ex] (5.5,0) -- (5.5,8.5) -- (10,8.5);
\filldraw (0.5,7.5) circle (.5ex);
\draw[line width = .2ex] (0.5,0) -- (0.5,7.5) -- (10,7.5);
\filldraw (1.5,6.5) circle (.5ex);
\draw[line width = .2ex] (1.5,0) -- (1.5,6.5) -- (10,6.5);
\filldraw (7.5,5.5) circle (.5ex);
\draw[line width = .2ex] (7.5,0) -- (7.5,5.5) -- (10,5.5);
\filldraw (8.5,4.5) circle (.5ex);
\draw[line width = .2ex] (8.5,0) -- (8.5,4.5) -- (10,4.5);
\filldraw (2.5,3.5) circle (.5ex);
\draw[line width = .2ex] (2.5,0) -- (2.5,3.5) -- (10,3.5);
\filldraw (4.5,2.5) circle (.5ex);
\draw[line width = .2ex] (4.5,0) -- (4.5,2.5) -- (10,2.5);
\filldraw (9.5,1.5) circle (.5ex);
\draw[line width = .2ex] (9.5,0) -- (9.5,1.5) -- (10,1.5);
\filldraw (6.5,0.5) circle (.5ex);
\draw[line width = .2ex] (6.5,0) -- (6.5,0.5) -- (10,0.5);
\end{tikzpicture}
\hspace{1.5em} 
\begin{tikzpicture}[scale=.35]
\draw (0,0) rectangle (10,10);

 \draw (0,10) rectangle (3,9);
 \draw (1,10) -- (1,9);
 \draw (2,10) -- (2,9);
 
\draw (4,6) rectangle (5,4);
 \draw (4,5)--(5,5);

\draw (6,5) rectangle (7,4);

\draw (6,3) rectangle (7,2);

\put(3,92){\scriptsize{$1$}}
\put(13,92){\scriptsize{$2$}}
\put(23,92){\scriptsize{$3$}}

\put(43,52){\scriptsize{$5$}}
\put(43,42){\scriptsize{$6$}}
\put(63,42){\scriptsize{$7$}}

\put(63,22){\scriptsize{$8$}}

\filldraw (3.5,9.5) circle (.5ex);
\draw[line width = .2ex] (3.5,0) -- (3.5,9.5) -- (10,9.5);
\filldraw (0.5,8.5) circle (.5ex);
\draw[line width = .2ex] (0.5,0) -- (0.5,8.5) -- (10,8.5);
\filldraw (1.5,7.5) circle (.5ex);
\draw[line width = .2ex] (1.5,0) -- (1.5,7.5) -- (10,7.5);
\filldraw (2.5,6.5) circle (.5ex);
\draw[line width = .2ex] (2.5,0) -- (2.5,6.5) -- (10,6.5);
\filldraw (5.5,5.5) circle (.5ex);
\draw[line width = .2ex] (5.5,0) -- (5.5,5.5) -- (10,5.5);
\filldraw (7.5,4.5) circle (.5ex);
\draw[line width = .2ex] (7.5,0) -- (7.5,4.5) -- (10,4.5);
\filldraw (4.5,3.5) circle (.5ex);
\draw[line width = .2ex] (4.5,0) -- (4.5,3.5) -- (10,3.5);
\filldraw (8.5,2.5) circle (.5ex);
\draw[line width = .2ex] (8.5,0) -- (8.5,2.5) -- (10,2.5);
\filldraw (6.5,1.5) circle (.5ex);
\draw[line width = .2ex] (6.5,0) -- (6.5,1.5) -- (10,1.5);
\filldraw (9.5,0.5) circle (.5ex);
\draw[line width = .2ex] (9.5,0) -- (9.5,0.5) -- (10,0.5);
\end{tikzpicture}
\hspace{1.5em} 
\begin{tikzpicture}[scale=.35]
\draw (0,0) rectangle (10,10);

 \draw (0,10) rectangle (3,8);
\draw (0,9)--(3,9);
\draw (1,10)--(1,8);
\draw (2,10)--(2,8);

\draw (4,8) rectangle (5,9);

\draw (2,6) rectangle (3,4);
 \draw (2,5)--(3,5);
 
\draw (4,6) rectangle (5,4);
 \draw (4,5)--(5,5);

\draw (6,6) rectangle (7,4);
 \draw (6,5)--(7,5);

\draw (6,1) rectangle (7,2);

\put(0.5,91.5){$+$}
\put(10.5,91.5){$+$}
\put(20.5,91.5){$+$}

\put(40.5,82){$+$}

\put(40.5,52){$+$}
\put(60.5,52){$+$}
\put(60.5,42){$+$}

\filldraw (3.5,9.5) circle (.5ex);
\draw[line width = .2ex] (3.5,0) -- (3.5,9.5) -- (10,9.5);
\filldraw (5.5,8.5) circle (.5ex);
\draw[line width = .2ex] (5.5,0) -- (5.5,8.5) -- (10,8.5);
\filldraw (0.5,7.5) circle (.5ex);
\draw[line width = .2ex] (0.5,0) -- (0.5,7.5) -- (10,7.5);
\filldraw (1.5,6.5) circle (.5ex);
\draw[line width = .2ex] (1.5,0) -- (1.5,6.5) -- (10,6.5);
\filldraw (7.5,5.5) circle (.5ex);
\draw[line width = .2ex] (7.5,0) -- (7.5,5.5) -- (10,5.5);
\filldraw (8.5,4.5) circle (.5ex);
\draw[line width = .2ex] (8.5,0) -- (8.5,4.5) -- (10,4.5);
\filldraw (2.5,3.5) circle (.5ex);
\draw[line width = .2ex] (2.5,0) -- (2.5,3.5) -- (10,3.5);
\filldraw (4.5,2.5) circle (.5ex);
\draw[line width = .2ex] (4.5,0) -- (4.5,2.5) -- (10,2.5);
\filldraw (9.5,1.5) circle (.5ex);
\draw[line width = .2ex] (9.5,0) -- (9.5,1.5) -- (10,1.5);
\filldraw (6.5,0.5) circle (.5ex);
\draw[line width = .2ex] (6.5,0) -- (6.5,0.5) -- (10,0.5);
\end{tikzpicture}
\hspace{1.5em} 
\begin{tikzpicture}[scale=.35]
\draw (0,0) rectangle (10,10);

 \draw (0,10) rectangle (3,8);
\draw (0,9)--(3,9);
\draw (1,10)--(1,8);
\draw (2,10)--(2,8);

\draw (4,8) rectangle (5,9);

\draw (2,6) rectangle (3,4);
 \draw (2,5)--(3,5);
 
\draw (4,6) rectangle (5,4);
 \draw (4,5)--(5,5);

\draw (6,6) rectangle (7,4);
 \draw (6,5)--(7,5);

\draw (6,1) rectangle (7,2);

\put(40.5,82){$+$}

\put(0.5,91.5){$+$}
\put(10.5,91.5){$+$}
\put(20.5,91.5){$+$}

\put(20.5,52){$+$}

\put(20.5,42){$+$}
\put(60.5,52){$+$}
\put(60.5,42){$+$}

\filldraw (3.5,9.5) circle (.5ex);
\draw[line width = .2ex] (3.5,0) -- (3.5,9.5) -- (10,9.5);
\filldraw (5.5,8.5) circle (.5ex);
\draw[line width = .2ex] (5.5,0) -- (5.5,8.5) -- (10,8.5);
\filldraw (0.5,7.5) circle (.5ex);
\draw[line width = .2ex] (0.5,0) -- (0.5,7.5) -- (10,7.5);
\filldraw (1.5,6.5) circle (.5ex);
\draw[line width = .2ex] (1.5,0) -- (1.5,6.5) -- (10,6.5);
\filldraw (7.5,5.5) circle (.5ex);
\draw[line width = .2ex] (7.5,0) -- (7.5,5.5) -- (10,5.5);
\filldraw (8.5,4.5) circle (.5ex);
\draw[line width = .2ex] (8.5,0) -- (8.5,4.5) -- (10,4.5);
\filldraw (2.5,3.5) circle (.5ex);
\draw[line width = .2ex] (2.5,0) -- (2.5,3.5) -- (10,3.5);
\filldraw (4.5,2.5) circle (.5ex);
\draw[line width = .2ex] (4.5,0) -- (4.5,2.5) -- (10,2.5);
\filldraw (9.5,1.5) circle (.5ex);
\draw[line width = .2ex] (9.5,0) -- (9.5,1.5) -- (10,1.5);
\filldraw (6.5,0.5) circle (.5ex);
\draw[line width = .2ex] (6.5,0) -- (6.5,0.5) -- (10,0.5);
\end{tikzpicture}
\]
\end{example}
 Defined by Woo--Yong \cite{WooYongGrobner}, the \mydef{unspecialized Grothendieck polynomial} is 
 \begin{equation}
 \label{eq:unspecgroth}
     \mathfrak G_{v,w}(\mathbf t):=\sum_{P\in\KPipes(v,w)}(-1)^{|P|-{\sf inv}(w)} \prod_{(i,j)\in P}t_{ij}.
 \end{equation}
 By setting $v=w_0\in S_n$ and specializing variables $t_{ij}$, these unspecialized Grothendieck polynomials recover the {double Grothendieck polynomials} of \cite{LS82}. 
Note that we follow the conventions of \cite{RRW} for $\mathfrak G_{v,w}(\mathbf t)$, which differ from those in \cite{WooYongGrobner}.

\subsection{Skew Excited Young Diagrams}\label{sec:skewEYD}
A permutation $u\in \Sym_n$ is \mydef{$321$-avoiding} if there does not exist a \mydef{$321$ pattern} in $u$, \emph{i.e.}, indices
$i<j<k$ such that $u_k<u_j<u_i$. 
For example, $u=1\underline{7}2\underline{5}83\underline{4}6$ is not $321$-avoiding; the underlined entries form a $321$ pattern in $u$. Let $\avoid{n}:=\{u\in S_n \, : \, u \text{ is } 321\text{-avoiding}\}$. A permutation $u\in \Sym_n$ is \mydef{Grassmannian} if there exists at most one $i\in[n-1]$ such that $u_i>u_{i+1}$. Grassmannian permutations form a subset of $321$-avoiding permutations.

For $u\in S_n$, let \[\phi_u:\{P\subseteq D(u)\}\rightarrow \{S\subseteq [n]^2\}\] be the map which deletes all empty rows and columns of $D(u)$ from $P\subseteq D(u)$, shifting remaining columns left and remaining rows up.

\begin{proposition}\cite[Proposition 2.2.13]{Manivel}\label{prop:321-graph}
    We have $u\in \avoid{n}$ if and only if $\codePD{u}:=\phi_u(D(u))$ is a skew Young diagram $\lambda / \mu$ for some partitions $\mu\subseteq\lambda$.
\end{proposition}

Note that our conventions for drawing Young diagrams reflect diagrams in English notation across the vertical axis.
Throughout this subsection, assume $v\geq w$, where $v,w\in \avoid{n}$. 

Let $D^{NE}(v,w)\subseteq D(v)$ be the boxes corresponding to the earliest subsequence ${\sf word}(P)$ of ${\sf word}(D(v))$ in $D(v)$ for $P\in \Pipes(v,w)$. Since $w\in \avoid{n}$, no braid moves are required to connect reduced words of $w$, so it is clear $D^{NE}(v,w)$ exists.

Define $\maxexcited{(v,w)}:=\phi_v(D^{NE}(v,w))$. 
We visualize $D\subseteq\codePD{v}$ by filling $(i,j)\in D$ with $+$'s and call $D$ a \mydef{diagram} in $\codePD{v}$.

\begin{example}\label{ex:SEYDexSetup}
    Recall $v,w$ as well as $P,P'$ from Example~\ref{ex:pipeSubword}. The left picture below is $\codePD{v}$. 
    Note that $P=D^{NE}(v,w)$, so the middle diagram below is $\phi_v(P)=\maxexcited{(v,w)}$.
    The rightmost diagram is $\phi_v(P')$:
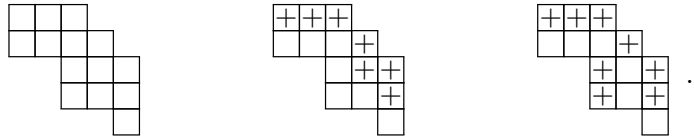
    \[\begin{picture}(260,35)
\put(0,20){\ytableausetup
{boxsize=0.8em}
{\begin{ytableau}
 \ &  \ &  \ \\
 \ &  \ & \ & \ \\
 \none &  \none & \ & \  & \ \\
  \none &  \none & \ & \  & \ \\
    \none &  \none & \none & \none  & \
\end{ytableau}}}
\put(100,20){\ytableausetup
{boxsize=0.8em}
{\begin{ytableau}
 + &  + &  + \\
 \ &  \ & \ & + \\
 \none &  \none & \ & +  & + \\
  \none &  \none & \ & \  & + \\
    \none &  \none & \none & \none  & \
\end{ytableau}}}
\put(200,20){\ytableausetup
{boxsize=0.8em}
{\begin{ytableau}
 + &  + &  + \\
 \ &  \ & \ & + \\
 \none &  \none & + & \  & + \\
  \none &  \none & + & \  & + \\
    \none &  \none & \none & \none  & \
\end{ytableau}}}
\end{picture}.
\]
\end{example}

An \mydef{excited move} on $D\subseteq\codePD{v}$ is the operation on a $2\times 2$ subsquare of $\codePD{v}$ such that
\begin{equation*}
    \ytableausetup{boxsize=0.8em}
\begin{ytableau}
\, & + \\
& \ 
\end{ytableau}
\hspace{1em}
 \raisebox{-.2em}{$\mapsto$}
 \hspace{1em}
\begin{ytableau}
\ &\\
+ &
\end{ytableau} \ .
\end{equation*}
We say this excited move was made at ${\sf b}\in D$, where ${\sf b}$ is the box containing the plus on the left-hand side above.
 For this move to occur, the subsquare must be contained in $\codePD{v}$.
Let 
$\skewexcited(v,w)$ denote the set of $D\subseteq \codePD{v}$ which can be computed through sequential applications of excited moves on $\maxexcited{(v,w)}$. We call a diagram $D\in \skewexcited(v,w)$ a \mydef{skew excited Young diagram} for $v,w$. For $v,w\in S_n$ Grassmannian, $\skewexcited(v,w)$ are ordinary excited Young diagrams, which arise in the study of vexillary matrix Schubert varieties \cite{KMY} as well as the equivariant cohomology and $K$-theory of the Grassmannian \cite{Graham.Kreiman,IN,KreimanEYD}.

A \mydef{K-excited move} on $D\subseteq \codePD{v}$ is the operation on a $2\times 2$ subsquare of $\codePD{v}$ such that
\begin{equation*}
    \ytableausetup{boxsize=0.8em}
\begin{ytableau}
\ & +\\
\ &
\end{ytableau}
\hspace{1em}
 \raisebox{-.2em}{$\mapsto$}
 \hspace{1em}
\begin{ytableau}
\, & + \\
+ & \ \\
\end{ytableau} \ ,
\end{equation*}
where all cells pictured are contained in $\codePD{v}$.
Write $\kskewexcited(v,w)$ for the set of diagrams obtainable through sequential applications of excited and K-excited moves on $\maxexcited{(v,w)}$ in $\codePD{v}$. We say a diagram $D\in \kskewexcited(v,w)$ is a \mydef{K-skew excited Young diagram} for $v,w$.
Let $|D|$ denote the number of plusses in $D$. We say $D\in \kskewexcited(v,w)$ is \mydef{maximal} if $D'\in \kskewexcited(v,w)$ implies $|D'|\leq |D|$.

\begin{example}\label{ex:SEYDex}
    Continuing Example~\ref{ex:SEYDexSetup}, the left two diagrams are in $\skewexcited(v,w)$. The right two diagrams are both maximal diagrams in $\kskewexcited(v,w)$.
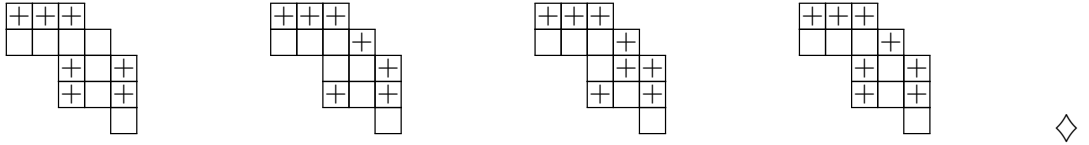
    \[\begin{picture}(350,35)
\put(0,20){\ytableausetup
{boxsize=0.8em}
{\begin{ytableau}
+ &  + &  + \\
 \ &  \ & \ & \ \\
 \none &  \none & + & \  & + \\
  \none &  \none & + & \  & + \\
    \none &  \none & \none & \none  & \
\end{ytableau}}}
\put(100,20){\ytableausetup
{boxsize=0.8em}
{\begin{ytableau}
 + &  + &  + \\
 \ &  \ & \ & + \\
 \none &  \none & \ & \  & + \\
  \none &  \none & + & \  & + \\
    \none &  \none & \none & \none  & \
\end{ytableau}}}
\put(200,20){\ytableausetup
{boxsize=0.8em}
{\begin{ytableau}
 + &  + &  + \\
 \ &  \ & \ & + \\
 \none &  \none & \ & +  & + \\
  \none &  \none & + & \  & + \\
    \none &  \none & \none & \none  & \
\end{ytableau}}}
\put(300,20){\ytableausetup
{boxsize=0.8em}
{\begin{ytableau}
 + &  + &  + \\
 \ &  \ & \ & + \\
 \none &  \none & + & \  & + \\
  \none &  \none & + & \  & + \\
    \none &  \none & \none & \none  & \
\end{ytableau}}}
\end{picture}
\]
\end{example}

\begin{proposition}\label{prop:EYDpipesbij}
   For $v\geq w$ where $v,w\in \avoid{n}$, the map $\phi_v$ restricted to $\KPipes{(v,w)}$ gives a bijection 
\begin{equation*}
\widetilde{\phi_v}:\KPipes(v,w)\rightarrow\kskewexcited(v,w)
\end{equation*} 
such that for $P\in \KPipes(v,w)$, $|P|=| \widetilde{\phi_v}(P)|$.
\end{proposition}  
\begin{proof}
For $D\subseteq[n]\times [n]$, a ladder move is the operation on a $2\times k$ strip in $D$ such that
\begin{equation*}
    \ytableausetup{boxsize=1.1em}
\begin{ytableau}
\ & + & {\scriptstyle{\text{$\cdots$}}} & +& +\\
\ & + & {\scriptstyle{\text{$\cdots$}}} & +&
\end{ytableau}
\hspace{1em}
 \raisebox{-.2em}{$\mapsto$}
 \hspace{1em}
\begin{ytableau}
\ & + & {\scriptstyle{\text{$\cdots$}}} & +& \ \\
+ & + & {\scriptstyle{\text{$\cdots$}}} & +&
\end{ytableau} \hspace{1em} \cup \hspace{1em}
\begin{ytableau}
\ & + & {\scriptstyle{\text{$\cdots$}}} & +& +\\
+ & + & {\scriptstyle{\text{$\cdots$}}} & +&
\end{ytableau} \ .
\end{equation*}
All cells above are contained in $[n]\times [n]$ and $k\geq 2$. Let 
\[S=\{D\subseteq D(v) \, : \, D \text{ obtained by applying ladder moves starting from } D^{NE}(v,w)\}.\]

Using \cite{BergeronBilley} and the subword complex interpretation of $\KPipes(v,w)$ as given in \cite[Section~3]{WooYongGrobner}, $S=\KPipes(v,w)$.
  By \cite[Theorem~4.1]{Gao} since $w\in\avoid{n}$, all ladder moves in this case are of the form
\begin{equation*}
    \ytableausetup{boxsize=1em}
\begin{ytableau}
\ & +\\
\ &
\end{ytableau}
\hspace{1em}
 \raisebox{-.2em}{$\mapsto$}
 \hspace{1em}
\begin{ytableau}
\, & \ \\
+ & \ \\
\end{ytableau} \hspace{1em} \cup
\hspace{1em}
\begin{ytableau}
\, & + \\
+ & \ \\
\end{ytableau}  \ .
\end{equation*}
Thus the statement follows by the definition of $\phi_v$.
\end{proof}

 \begin{corollary}\label{cor:unspGrSEYD} Suppose $v\geq w$ where $v,w\in \avoid{n}$. Then
\[\deg( \mathfrak G_{v,w}(\mathbf t))=\max\{|D| \, : \, D\in \kskewexcited(v,w)\}.\]
\end{corollary}
\begin{proof}
This follows by Proposition~\ref{prop:EYDpipesbij} and Equation~\eqref{eq:unspecgroth}.
\end{proof}

\begin{example}\label{ex:degVWeyd}
For $v,w$ as in Example~\ref{ex:SEYDex}, Corollary \ref{cor:unspGrSEYD} determines $\deg(\mathfrak G_{v,w}(\mathbf t))=8$.
\end{example}

\section{Castelnuovo--Mumford Regularity of Kazhdan--Lusztig varieties}\label{sec:geomBack}
In this section, we define Castelnuovo--Mumford regularity, $a$-invariants, and Kazhdan--Lusztig varieties. We then recall results of \cite{RRW} which relate the Castelnuovo--Mumford regularity of Kazhdan--Lusztig varieties to unspecialized Grothendieck polynomials.

\subsection{Castelnuovo--Mumford Regularity}
Let $S = \mathbb{C}[x_1,\ldots, x_n]$ be a polynomial ring with the standard grading, and let $I\subseteq S$ be a homogeneous ideal. 
 The \mydef{Hilbert series} of $S/I$ is a formal power series
\begin{equation*}
    H(S/I; t) = \sum_{k\in \mathbb{Z}}\text{dim}_{\mathbb{C}}((S/I)_{k}){t}^{k}
=\frac{K(S/I;t)}{(1-{t})^{n}}.
\end{equation*}
The $\mathbf{K}$\mydef{-polynomial} of $S/I$ is the numerator $K(S/I;{t})\in \mathbb{C}[t^{\pm 1}]$. 
A minimal free resolution of $S/I$ is the complex
\[
 0 \rightarrow \bigoplus_jS(-j)^{\beta_{l,j}(S/I)}\rightarrow \bigoplus_jS(-j)^{\beta_{l-1,j}(S/I)}\rightarrow \cdots \rightarrow \bigoplus_jS(-j)^{\beta_{0,j}(S/I)} \rightarrow S/I \rightarrow 0,
\]
where $l\leq n$ and $S(-j)$ is the free $S$-module with degree shifted by $j$. 
The \mydef{Castelnuovo--Mumford regularity} of $S/I$, written $\reg(S/I)$, is the statistic
\[
\reg(S/I):=\max\{j-i \, :  \,\beta_{i,j}(S/I)\neq 0\}.
\]
For $S/I$ Cohen--Macaulay,
\begin{equation}\label{eq:mainRegEquation}
\reg(S/I) = \text{deg }K(S/I;t) - \text{ht}_S I,
\end{equation}
where $\text{ht}_S I$ denotes the height of the ideal $I$. For more context, consult \cite[Lemma 2.5]{Benedetti.Varbaro}. 

The \mydef{$a$-invariant} of $S/I$, written $a(S/I)$, is the negative of the least degree of a generator of the graded canonical module of $S/I$, as defined by Goto--Watanabe \cite{GotoWat}. When $S/I$ is Cohen--Macaulay,
\begin{equation}\label{eq:aInvReg}
    a(S/I)=\reg(S/I)-d,
\end{equation}
where $d$ is the Krull dimension of $S/I$.

\subsection{Kazhdan--Lusztig varieties}
 We follow the conventions used in \cite{RRW}.
For $v\in S_n$, define $M^{(v)}=(m_{ij})$ to be the $n\times n$ matrix such that for $i,j\in[n]$,
\[m_{ij}:=
\begin{cases}
1 &\text{if } v_i=j, \\
z_{ij} & \text{if } (i,j)\in D(v),\\
0 & \text{ otherwise}.
\end{cases}\]
Let $\mathbb{C}[{\bf z}^v]:= \mathbb{C}[z_{ij}\, :  \,(i,j)\in D(v)]$. For $v\geq w$ where $v,w\in S_{n}$, the \mydef{Kazhdan--Lusztig ideal} $J_{v,w}\subseteq \mathbb{C}[{\bf z}^v]$ is defined by
\[
J_{v,w} := \langle {(\sf rank}_w(i,j)+1) \text{\,-\,minors}  \text{ in } M^{(v)}_{[i],[j]}\, :  \,(i,j)\in D(w)\rangle,
\]
where $M_{I,J}$ denotes the submatrix of $M$ with row indices in $I$ and column indices in $J$ for $I,J\subseteq[n]$. When $v\in\avoid{n}$, $J_{v,w}$ is homogeneous, see \cite[Footnote on pg.~25]{Knutson-Frob}. Additional cases for which $J_{v,w}$ is homogeneous can be found in \cite[Propositions~6.3 and~6.4]{Neye}, but no full characterization is known.

Let $B_+,B_-\subset \GL_n(\mathbb{C})$ denote the Borel and opposite Borel subgroups, respectively. As defined in \cite{WooYongSings}, the \mydef{Kazhdan--Lusztig variety} is the intersection of the \mydef{Schubert variety} $B_-\backslash \overline{B_- w B_+}\subseteq B_-\backslash \GL_n(\mathbb{C})$ with the \mydef{opposite Schubert cell} $B_-\backslash B_-vB_-$. The coordinate ring of this {Kazhdan--Lusztig variety} is $\mathbb{C}[{\bf z}^{v}]/J_{v,w}$. 
Using \cite[Lemma A.4]{KazhdanLusztig} and the fact that Schubert varieties are Cohen--Macaulay \cite{Fulton.Flags,KM,Ramanathan},
 $\mathbb{C}[{\bf z}^{v}]/J_{v,w}$ is Cohen--Macaulay.

As reformulated in \cite[Lemma 6.3]{RRW},
\begin{lemma}\cite[Theorem 4.5]{WooYongGrobner}\label{lem:321KFormula}
Let $v,w\in \avoid{n}$ where $v\geq w$.
Then 
\begin{equation*}
   K(\mathbb{C}[{\bf z}^v]/J_{v,w};t) =  \sum_{P\in \KPipes(v,w)}(-1)^{|P|-{\sf inv}(w)}(1-t)^{|P|}.
\end{equation*}
\end{lemma}

Combining Lemma~\ref{lem:321KFormula} with Equation~\eqref{eq:mainRegEquation} produces the following:
\begin{proposition}\cite[Proposition 6.4]{RRW}
\label{prop:pipeKLreg}
Let $v,w\in \avoid{n}$ where $v\geq w$.  Then
\begin{equation*}\label{eq:keyDegFormula}
   \deg K(\mathbb{C}[{\bf z}^v]/J_{v,w};t) = \deg \mathfrak{G}_{v,w}(\bf{t}).
\end{equation*}
Furthermore, the Castelnuovo--Mumford regularity of $\mathbb{C}[{\bf z}^v]/J_{v,w}$ is given by
\begin{equation*}
    \reg (\mathbb{C}[{\bf z}^v]/J_{v,w}) =   \deg \mathfrak{G}_{v,w}({\bf{t}})-{\sf inv}(w).
\end{equation*}
\end{proposition}

Applying this to $a$-invariants:
\begin{corollary}
\label{cor:pipeKLainv}
Let $v,w\in \avoid{n}$ where $v\geq w$.  
The $a$-invariant of $\mathbb{C}[{\bf z}^v]/J_{v,w}$ is 
\begin{equation*}
    a(\mathbb{C}[{\bf z}^v]/J_{v,w}) =   \deg \mathfrak{G}_{v,w}({\bf{t}})-{\sf inv}(v).
\end{equation*}
\end{corollary}
\begin{proof}
    This follows by Proposition~\ref{prop:pipeKLreg} combined with Equation~\eqref{eq:aInvReg} since $\mathbb{C}[{\bf z}^v]/J_{v,w}$ has dimension $d={\sf inv}(v)-{\sf inv}(w)$.
\end{proof}

For $S/I$ Cohen--Macaulay, the $a$-invariant is the lower bound for when its Hilbert function and Hilbert polynomial agree. 
Using Equations~\eqref{eq:mainRegEquation} and~\eqref{eq:aInvReg}, Stelzer and Yong \cite{StelzerYong} prove that all homogeneous Kazhdan--Lusztig varieties are Hilbertian, \emph{i.e.}, the Hilbert function and Hilbert polynomial of a Kazhdan--Lusztig variety agree at all non-negative integer values, excepting the $v=w$ case.

\section{Main Construction and Proofs}\label{sec:mainThmPf}

Assume $v\geq w$ where $v,w\in \avoid{n}$ throughout this section. In Section~\ref{sec:321321alg}, we give key technical constructions. In Section~\ref{sec:degProof}, we define the diagram $\Kzip(v,w)$ and relate $\Kzip(v,w)$ to $\deg \mathfrak{G}_{v,w}({\bf{t}})$. Section~\ref{sec:321321reg} contains proofs of Theorem~\ref{thm:321321IndReg} and Corollary~\ref{cor:321321IndAInv}.

\subsection{Main Construction}\label{sec:321321alg}
Note that we index $\codePD{v}$ using matrix indexing, where its upper leftmost box corresponds to $(1,1)$.  For a box ${\sf b}\in [n]^2$, we write ${\sf b}=({\sf b}(1),{\sf b}(2))$.  We say $D\subseteq [n]^2$ is a \mydef{diagonal} if each pair ${\sf b}, {\sf c}\in D$ satisfies ${\sf b}(1)\neq {\sf c}(1)$ and if ${\sf b}(1)<{\sf c}(1)$ then ${\sf b}(2)<{\sf c}(2)$. A \mydef{unit diagonal} $D$ in $[n]^2$ is a diagonal in $[n]^2$ such that if ${\sf b},{\sf c}\in D$, then ${\sf b}(1)-{\sf b}(2)={\sf c}(1)-{\sf c}(2)$. 

 We build an equivalence relation $\sim$ on boxes across a sequence of skew excited Young diagrams connected by excited moves.
 Consider a sequence of diagrams $(D_0,\ldots,D_t)$ such that $D_i\subseteq [n]^2$ for $ 0\leq i\leq t$, and 
 each $D_i$ is obtained from $D_{i-1}$ by one excited move at ${\sf b}_i\in D_{i-1}$ for $i\in[t]$. For each move,  we define  ${\sf c}\sim{\sf c}$ for ${\sf c}\in D_{i-1}\cap D_i$ and additionally ${\sf b}_i\sim{\sf b}_i+(1,-1)$ where ${\sf b}_i\in D_{i-1}$ and ${\sf b}_i+(1,-1)\in D_i$. 
We allow the possibility of null moves $D_{i-1}=D_{i}$ for reflexivity.
Note this relation is well-defined across different choices of paths from $D_0$ to $D_t$ since the set of moves to be applied to obtain $D_t$ from $D_0$ is unique.
 
Let $D\subseteq[n]^2$ and take ${\sf b},{\sf c}\in D$. We define a \mydef{SW lattice path} from ${\sf b}$ to ${\sf c}$ in $D$ as a sequence $({\sf b}_0,{\sf b}_1,\dots,{\sf b}_{t})$, where 
${\sf b}_0={\sf b}$, ${\sf b}_t={\sf c}$, and ${\sf b}_i\in \{{\sf b}_{i-1}+(1,0),{\sf b}_{i-1}+(0,-1)\}\cap D$ for $i\in[t]$ and $t\in\mathbb{Z}_{\geq 0}$. 
 Now suppose ${\sf b}, {\sf c}\in D$, where $D\in\skewexcited{(v,w)}$ such that ${\sf b}$ is the topmost then the rightmost among ${\sf b}$ and ${\sf c}$.
 Define $\gap({\sf b}, {\sf c})=\gap({\sf c}, {\sf b})$ as the minimal $g\in\mathbb{Z}_{\geq 0}$ such that 
 there exists a diagram $D'$ obtainable via successive excited moves on $D$ where the following hold:
 \begin{itemize}
     \item ${\sf b}'={\sf b}+(g,-g)\in D'$,
     \item ${\sf b}'\sim {\sf b}$, and
     \item there exists a SW lattice path in $D'$ from ${\sf b}'$ to ${\sf c}$.
 \end{itemize}
 If no such $g$ exists,
 we define $\gap({\sf b}, {\sf c})=\infty$.
 We say boxes ${\sf b}, {\sf c}\in D$ lie within the same \mydef{connected component} of $D$ if there exists a lattice path in $D$ (\emph{i.e.}, a sequence of boxes in $D$ sharing horizontal or vertical edges) from ${\sf b}$ to $ {\sf c}$.

For ${\sf b}\in \maxexcited{(v,w)}$, 
consider the  rightmost box ${\sf a}\in \maxexcited{(v,w)}$ 
such that ${\sf a}+(0,1)={\sf b}'$ for some ${\sf b}'\sim {\sf b}$.
If such a box ${\sf a}$ exists, 
define 
$\WB{{\sf b}}:={\sf a}$. Otherwise take $\WB{{\sf b}}:={\sf b}$.
Similarly, 
consider the topmost box ${\sf c}\in \maxexcited{(v,w)}$ 
such that ${\sf c}-(1,0)={\sf b}'$ for some ${\sf b}'\sim {\sf b}$.
If such a box ${\sf c}$ exists, 
define 
$\SB{{\sf b}}:={\sf c}$. Otherwise take $\SB{{\sf b}}:={\sf b}$.
If $\SB{{\sf b}}=\WB{{\sf b}}={\sf b}$, we say ${\sf b}$ is a \mydef{corner} in $\maxexcited{(v,w)}$.

\begin{example}\label{ex:DtopCC}
 Consider  
 $\maxexcited{(v,w)}$ below for certain $v, w \in\avoid{14}$\footnote{Here $v= (5,10,11,12,1,13,2,3,14,4,6,7,8,9)$, $w=(1,2,5,6,10,3,4,11,7,13,8,9,12,14)$.}.
 Then $\maxexcited{(v,w)}$ has three connected components.
     \[\ytableausetup
{boxsize=0.8em}
{\begin{ytableau}
  \ &  \ &  + &  + \\
 \ &  \ &  + &  + &  \  &  + &  + &  + \\
 \ &  \ &  + &  + &  \  &  + &  + &  + \\
 \ &  \ &  \ &  \ &  \  &  \ &  \ &  \ \\
 \none &  \ &  \ &  \ &  \  &  + &  + &  + \\
  \none &  \none &  \none &    \ & \ &  \ &  \ &  \ 
\end{ytableau}}
\]
We find 
$\gap((3,4),(2,6))=1$, $\gap((3,4),(2,4))=0$, and  $\gap((3,4),(3,6))=\infty$.
We can compute  $\WB{{(2,6)}}=(3,4)$ and $\SB{(2,6)}=(3,6)$. Similarly, we determine $\WB{{(3,7)}}=(3,6)$ and $\SB{(3,7)}=(5,6)$.
 \end{example}

Suppose $D\subseteq\codePD{v}$.
Among the diagonals of maximal length in each connected component of $D$, take the bottommost then leftmost diagonals in each connected component. Let $Q(D)$ denote the set of the boxes in these diagonals.
Let $\Lambda(D)\subseteq Q(D)$ be the subset of boxes where ${\sf b}\in   Q(D)-\Lambda(D)$ if and only if there exists ${\sf c}\in  \Lambda(D)$ such that the following hold:
\begin{itemize}
    \item $\gap({{\sf b}}, {{\sf c}})<\infty$, and 
    \item ${\sf b}$ is the topmost then the rightmost among ${\sf b}$ and ${\sf c}$.
\end{itemize}
For ${\sf b}\in   Q(D)-\Lambda(D)$, define ${\sf trail}({\sf b}):=\gap({{\sf b}}, {{\sf c}})$, where ${\sf c}$ is the rightmost box in $\Lambda(D)$ such that $\gap({{\sf b}}, {{\sf c}})<\infty$.

\begin{example}\label{ex:badTrail}
 We continue with $v,w$ as in Example~\ref{ex:DtopCC}. In $D=\maxexcited{(v,w)}$, we have shaded the $Q(D)$ and bolded ${\Lambda}(D)\subseteq Q(D)$.
     \[\ytableausetup
{boxsize=0.8em}
{\begin{ytableau}
  \ &  \ &  + &  + \\
 \ &  \ & *(lightgray!50)  \Plus &  + &  \  &  *(lightgray!50) + &  + &  + \\
 \ &  \ &  + & *(lightgray!50)  \Plus &  \  &  + & *(lightgray!50)  + &  + \\
 \ &  \ &  \ &  \ &  \  &  \ &  \ &  \ \\
 \none &  \ &  \ &  \ &  \  & *(lightgray!50)  \Plus &  + &  + \\
  \none &  \none &  \none &    \ & \ &  \ &  \ &  \ 
\end{ytableau}}
\]
We see the boxes $(2,6), (3,7)\in Q(D)-\Lambda(D)$. Then ${\sf trail}((2,6))=\gap((2,6), (3,4))=1$ and  ${\sf trail}((3,7))=\gap((3,7), (5,6))=1$.
 \end{example}

 Let $\dTab(D)$ denote the set of tableaux on a diagram $D\subseteq[n]^2$ with non-negative entries, \emph{i.e.}, the set of maps $D\rightarrow \mathbb{Z}_{\geq 0}$.

\begin{construction}\label{constr:tabT}
We define a distinguished tableau $\Tab{v,w}\in\dTab(D)$ for $D=\maxexcited{(v,w)}$.
For a corner ${\sf c}\in D$, define $\Tab{v,w}({\sf c})$ such that
\begin{align*}
    \Tab{v,w}({\sf c}):=\max\Big(\{0\}\cup \big\{k\in[n] \, : \, 
    &{\sf c}+(k',-k'),{\sf c}+(k',1-k'),\\
    &{\sf c}+(k'-1,-k')\in \codePD{v}- D \text{ for each } k'\in [k]\big\}\Big).
\end{align*}   
Note this is precisely the maximal number of excited moves that may be applied to ${\sf c}$.
Otherwise if ${\sf b}\in D$ is not a corner,  define 
\[
\Tab{v,w}({{\sf b}}):=
\begin{cases}
    \min{\big(\Tab{v,w}(\WB{{\sf b}})+\gap(\WB{{\sf b}}, {{\sf b}}),\, \Tab{v,w}(\SB{{\sf b}})+\gap(\SB{{\sf b}}, {{\sf b}})\big)}& \text{if } \SB{{\sf b}}\neq{\sf b} \text{ and }\WB{{\sf b}}\neq{\sf b},\\
    \Tab{v,w}(\WB{{\sf b}})+\gap(\WB{{\sf b}}, {{\sf b}})& \text{if } \SB{{\sf b}}={\sf b},\\
     \Tab{v,w}(\SB{{\sf b}})+\gap(\SB{{\sf b}}, {{\sf b}}) & \text{if } \WB{{\sf b}}={\sf b}.\\
\end{cases}
\]
\end{construction}
\noindent For $U\in\dTab(\maxexcited{(v,w)})$, let $\Pos{U}\subseteq \maxexcited{(v,w)}$ be defined such that:
\[\Pos{U}=\{{\sf b}\in \maxexcited{(v,w)} \, : \, U({\sf b})>0\}.\]
 \begin{construction}\label{constr:tabTi}
 Set $D=\maxexcited{(v,w)}$.
 We iteratively construct tableaux $\fTab{v,w}^{(i)}\in \dTab({ D})$ and subsets $D^{(i)}\subset D$ for $0\leq i\leq \ell$ for some $\ell\in\mathbb{Z}_{\geq 0}$. 
  Let $\fTab{v,w}^{(-1)}:=\Tab{v,w}$ and $D^{(-1)}:=D$.
  Let $Q^{(i)}:=Q(D^{(i-1)})$ and $\Lambda^{(i)}:=\Lambda(D^{(i-1)})$.

Partition $\Lambda^{(i)}=\bigsqcup_j \Lambda^{(i)}_j$ into equivalence classes defined by the transitive closure of relating boxes ${\sf d},{\sf d}'$ when there exists a corner ${\sf c }\in D$ such that $\gap({\sf d}, {\sf c})=\gap({\sf d}', {\sf c})<\infty$.
Using this, we define $\fTab{v,w}^{(i)}\in \dTab({ D})$ such that
  \[
\fTab{v,w}^{(i)}({{\sf b}}):=
\begin{cases}
    \fTab{v,w}^{(i-1)}({{\sf b}}) & \text{ if } {\sf b}\in Q^{(0)}\cup\bigcup_{j<i} \Lambda^{(j)} \text{ and } i>0,\\
    {\sf trail}({\sf b}) & \text{ if } {\sf b}\in Q^{(0)}-\Lambda^{(0)}\text{ and } i=0,\\
    \min_{{\sf d}\in \Lambda^{(i)}_j}\{\fTab{v,w}^{(i-1)}({{\sf d}})\} & \text{ if } {\sf b}\in \Lambda^{(i)}_j,\\
   \fTab{v,w}^{(i-1)}({{\sf b}})-\fTab{v,w}^{(i)}({{\sf d}}) & \text{ if } \gap({\sf d}, {\sf b})=0, {\sf d}\in  \Lambda^{(i)}-\{{\sf b}\},  {\sf b}(1)\geq{\sf d}(1), {\sf b}(2)\leq {\sf d}(2),\\
   0 & \text{ otherwise.}
\end{cases}
\]
By the construction of the partition $\Lambda^{(i)}=\bigsqcup_j \Lambda^{(i)}_j$, the tableau $\fTab{v,w}^{(i)}$ is well-defined.
Set 
\[D^{(i)}=\{{\sf b}\in \Pos{(\fTab{v,w}^{(i)})} \, : \, \gap({\sf b}, {\sf d})=0, \text{ for }  {\sf d}\in\Lambda^{(i)} \text{ with } {\sf b}\neq{\sf d}, {\sf b}(1)\geq{\sf d}(1), \text{ and }{\sf b}(2)\leq {\sf d}(2)\}.\]

Note that if ${\sf b}\in D^{(i)}$ then $\SB{{\sf b}},\WB{{\sf b}}\in D^{(i)}$.
The process terminates when $D^{(i)}=\emptyset$. Let $\ell$ denote the index such that $D^{(\ell)}=\emptyset$.
Take  $\fTab{v,w}:=\fTab{v,w}^{(\ell)}$. Set $\mdCR{v,w}= Q^{(0)}\cup\bigcup_{i=1}^\ell\Lambda^{(i)}$, so we see $\Pos{(\fTab{v,w})}\subseteq \mdCR{v,w}$. By the definitions of $\Lambda^{(i)}$ and $D^{(i)}$, we have $\Lambda^{(i)}=Q^{(i)}$ for $i>0$.
\end{construction}

Define \[\CM{(v,w)}=\Big(\sum_{i=0}^{\ell}\sum_{{\sf d}\in \Lambda^{(i)}}\fTab{v,w}({\sf d})\Big)+\sum_{{\sf d}\in Q^{(0)}-\Lambda^{(0)}}{\sf trail}({\sf d}).\]

\begin{example}\label{ex:badExTrailFinish}
  We now apply Construction~\ref{constr:tabTi}, continuing with $v,w$ as in Example~\ref{ex:badTrail}.
     To the left is $\Tab{v,w}=\fTab{v,w}^{(-1)}$ with the entries in $\Lambda^{(0)}$ bolded. To the right is $\fTab{v,w}^{(0)}$. 
 \ytableausetup{boxsize=1em}
\begin{gather*}
\begin{ytableau}
  \ &  \ &  1 &  1 \\
  \ &  \ &   \bf{1} &  1 & \ &  2 &  2 &  2 \\
    \ &  \ &  1 &   \bf{1} & \ &  2 &  2 &  2 \\
  \ &  \ &  \ &  \ & \ &  \ &  \ &  \ \\
  \none &  \ &  \ &  \ & \ &   \bf{1} &  1 &  1 \\  
   \none &  \none &  \none &  \ & \ &  \ &  \ &  \ \\  
\end{ytableau}
\qquad\qquad
\begin{ytableau}
 \ &  \ &  0 &  0 \\
  \ &  \ &  1 &  0 & \ &  1 &  0 &  0 \\
    \ &  \ &  0 &  1 & \ &  0 &  1 &  0 \\
  \ &  \ &  \ &  \ & \ &  \ &  \ &  \ \\
  \none &  \ &  \ &  \ & \ &  1 &  0 &  0 \\  
   \none &  \none &  \none &  \ & \ &  \ &  \ &  \ \\  
\end{ytableau}
\end{gather*}
Then ${ D}^{(0)}=\emptyset$, so $\fTab{v,w}^{(0)}=\fTab{v,w}$.
We compute 
    \[
\CM{(v,w)}=\sum_{{\sf d}\in \Lambda^{(0)}}\fTab{v,w}({\sf d})+\sum_{{\sf d}\in Q^{(0)}-\Lambda^{(0)}}{\sf trail}({\sf d})=(1+1+1)+(1+1)=5.
    \]
In this case, we have $\mdCR{v,w}=\Lambda^{(0)}\cup (Q^{(0)}-\Lambda^{(0)})$ and $\Lambda^{(0)}\subsetneq Q^{(0)}$.    
\end{example}

We say the pair $v,w$ is \mydef{\UD} if the following hold:
\begin{enumerate}
    \item[(i)] if ${\sf b}, {\sf c}\in\maxexcited{(v,w)}$ such that ${\sf b}={\sf c}+(k,0)$, then ${\sf c}+(k',0)\in \maxexcited{(v,w)}$ for each $0\leq k'\leq k$.
    \item[(ii)] if ${\sf b},{\sf c}\in\maxexcited{(v,w)}$ lie weakly to the lower left of $Q^{(0)}$, then $\gap({\sf b}, {\sf c})\in\{0,\infty\}$.
\end{enumerate}

\begin{remark}\label{rmk:Con}
  In Construction~\ref{constr:tabTi} if $v,w$ are {\UD},  we have $Q^{(i)}=\Lambda^{(i)}$ for each $0\leq i\leq \ell$.  
\end{remark}

\begin{example}\label{ex:upperDec}
Looking to Example~\ref{ex:DtopCC}, those $v,w$ are not {\UD}. For example, ${\sf b}=(5,6)$ and ${\sf c}=(3,6)$ violate condition (i). Further, ${\sf b}=(3,4)$ and ${\sf c}=(2,6)$ violate (ii) since ${\sf b},{\sf c}\in Q^{(0)}$ but $\gap({\sf b},{\sf c})=1\not\in \{0,\infty\}$. However, it is straightforward to check that $v,w$ from Example~\ref{ex:introThm} are {\UD}.
\end{example}

For added clarity, we include a larger example for Construction~\ref{constr:tabTi}:
\begin{example}\label{ex:TvwConstr}
Consider  
 $\maxexcited{(v,w)}$ below for certain $v, w \in\avoid{22}$\footnote{The corresponding permutations are $v= (5,6,9,1,10,13,16,17,20,21,2,3,4,7,8,22,11,12,14,15,18,19)$ and
 $w=(1,5,6,2,9,3,10,13,4,14,7,16,8,11,12,17,15,20,21,18,19,22)$.} that are {\UD}.
     \[\begin{picture}(100,75)
\put(0,65){\Small\ytableausetup
{boxsize=0.72em}
{
\begin{ytableau}
  \ &  + &  + &  + \\
  \ &  + &  + &  + \\
  \ &  \ &  + &  + &  + &  + \\
  \none &  \ &  \ &  + &  + &  +  \\
  \none &  \ &  \ &  + &  + &  + &  + &  +  \\
  \none &  \ &  \ &  \ &  + &  + &  + &  + &  \ &  + \\
  \none &  \ &  \ &  \ &  \ &  + &  + &  + &  \ &  + \\
  \none &  \ &  \ &  \ &  \ &  \ &  \ &  \ &  \ &  \ &  + &  + \\
  \none &  \ &  \ &  \ &  \ &  \ &  \ &  \ &  \ &  \ &  + &  +  \\
  \none &  \none &  \none &  \none &  \none &  \none &  \ &  \ &  \ &  \ &  \ &  \
\end{ytableau}}}
\end{picture}
\]
In the diagrams below, we have $\fTab{v,w}^{(i)}$ for $-1\leq i\leq 2$, where we bold $Q^{(i+1)}=\Lambda^{(i+1)}$ when $-1\leq i\leq 1$. Additionally, we shade boxes in $\bigcup_{j\leq i}\Lambda^{(j)}$ in each $\fTab{v,w}^{(i)}$.
Below, the upper left is $\Tab{v,w}=\fTab{v,w}^{(-1)}$, and to its right is $\fTab{v,w}^{(0)}$. In the lower row are $\fTab{v,w}^{(1)}$ and  $\fTab{v,w}^{(2)}$, respectively.
\ytableausetup{boxsize=0.9em}
\begin{gather*}
\begin{ytableau}
  \ &  \bf{1} &  1 &  1 \\
  \ &  1 &  \bf{1} &  1 \\
  \ &  \ &  1 &  \bf{1} &  1 &  1 \\
  \none &  \ &  \ &  2 &  \bf{2} &  2 \\
  \none &  \ &  \ &  2 &  2 &  \bf{2} &  2 &  2 \\
  \none &  \ &  \ &  \ &  3 &  2 &  \bf{2} &  2 &  \ &  3 \\
  \none &  \ &  \ &  \ &  \ &  2 &  2 &  \bf{2} &  \ &  \bf{3} \\
  \none &  \ &  \ &  \ &  \ &  \ &  \ &  \ &  \ &  \ &  \bf{1} &  1 \\
  \none &  \ &  \ &  \ &  \ &  \ &  \ &  \ &  \ &  \ &  1 &  \bf{1} \\
  \none &  \none &  \none &  \none &  \none &  \none &  \ &  \ &  \ &  \ &  \ &  \
\end{ytableau}
\qquad\qquad
\begin{ytableau}
  \ &  *(gray!30)1 &  0 &  0 \\
  \ &  0 &  *(gray!30)1 &  0 \\
  \ &  \ &  0 & *(gray!30) 1 &  0 &  0 \\
  \none &  \ &  \ &  \bf{1} &  *(gray!30)1 &  0 \\
  \none &  \ &  \ &  1 &  \bf{1} & *(gray!30) 1 &  0 &  0 \\
  \none &  \ &  \ &  \ &  2 &  \bf{1} &  *(gray!30)1 &  0 &  \ &  0 \\
  \none &  \ &  \ &  \ &  \ &  1 &  \bf{1} & *(gray!30) 1 &  \ &  *(gray!30)3 \\
  \none &  \ &  \ &  \ &  \ &  \ &  \ &  \ &  \ &  \ &  *(gray!30)1 &  0 \\
  \none &  \ &  \ &  \ &  \ &  \ &  \ &  \ &  \ &  \ &  0 & *(gray!30) 1 \\
  \none &  \none &  \none &  \none &  \none &  \none &  \ &  \ &  \ &  \ &  \ &  \
\end{ytableau}
\\[1em]
\begin{ytableau}
  \ &  *(gray!30)1 &  0 &  0 \\
  \ &  0 &  *(gray!30)1 &  0 \\
  \ &  \ &  0 & *(gray!30) 1 &  0 &  0 \\
  \none &  \ &  \ &  *(gray!30)1 &  *(gray!30)1 &  0 \\
  \none &  \ &  \ &  0 & *(gray!30) 1 &  *(gray!30)1 &  0 &  0 \\
  \none &  \ &  \ &  \ &  \bf{1} &  *(gray!30)1 &  *(gray!30)1 &  0 &  \ &  0 \\
  \none &  \ &  \ &  \ &  \ &  0 &  1 &  *(gray!30)1 &  \ &  *(gray!30)3 \\
  \none &  \ &  \ &  \ &  \ &  \ &  \ &  \ &  \ &  \ &  *(gray!30)1 &  0 \\
  \none &  \ &  \ &  \ &  \ &  \ &  \ &  \ &  \ &  \ &  0 &  *(gray!30)1 \\
  \none &  \none &  \none &  \none &  \none &  \none &  \ &  \ &  \ &  \ &  \ &  \
\end{ytableau}
\qquad\qquad
\begin{ytableau}
 \ &  *(gray!30)1 &  0 &  0 \\
  \ &  0 &  *(gray!30)1 &  0 \\
  \ &  \ &  0 &  *(gray!30)1 &  0 &  0 \\
  \none &  \ &  \ &  *(gray!30)1 &  *(gray!30)1 &  0 \\
  \none &  \ &  \ &  0 &  *(gray!30)1 &  *(gray!30)1 &  0 &  0 \\
  \none &  \ &  \ &  \ & *(gray!30) 1 & *(gray!30) 1 &  *(gray!30)1 &  0 &  \ &  0 \\
  \none &  \ &  \ &  \ &  \ &  0 &  *(gray!30)1 & *(gray!30) 1 &  \ &  *(gray!30)3 \\
  \none &  \ &  \ &  \ &  \ &  \ &  \ &  \ &  \ &  \ &  *(gray!30)1 &  0 \\
  \none &  \ &  \ &  \ &  \ &  \ &  \ &  \ &  \ &  \ &  0 &  *(gray!30)1 \\
  \none &  \none &  \none &  \none &  \none &  \none &  \ &  \ &  \ &  \ &  \ &  \
\end{ytableau}
\end{gather*}

We see $D^{(2)}=\emptyset$, so $\fTab{v,w}^{(2)}=\fTab{v,w}$. 
Then
\[\CM{(v,w)}=\sum_{{\sf d}\in \Lambda^{(0)}}\fTab{v,w}^{(0)}({\sf d})+\sum_{{\sf d}\in \Lambda^{(1)}}\fTab{v,w}^{(1)}({\sf d})+\sum_{{\sf d}\in \Lambda^{(2)}}\fTab{v,w}^{(2)}({\sf d})=(7+ 3 +2)+4+1=17.
\]
\end{example}

\subsection{Applying Constructions~\ref{constr:tabT} and~\ref{constr:tabTi}}\label{sec:degProof}
We use the construction in the previous section to obtain the following result:

\begin{theorem}\label{thm:321321IndDeg}
    For $v\geq w$, where $v,w\in\avoid{n}$, we have \[\deg \mathfrak{G}_{v,w}({\bf{t}})\geq \CM{(v,w)}+{\sf inv}(w).\]
    If $v,w$ are {\UD}, equality is achieved.
\end{theorem}
The proof is given at the end of this section.
First, we need an additional construction and some technical preliminary results.
Consider the map $\Psi:\kskewexcited{(v,w)}\rightarrow\dTab(\maxexcited{(v,w)})$
 such that for $D\in\kskewexcited{(v,w)}$, the box ${\sf b}$ in $\Psi(D)$ is
 filled with $m$, where $m$ is the number of $K$-excited moves applied to 
\[{\sf b}',{\sf b}'+(1,-1),\ldots,{\sf b}'+(m-1,1-m)\]
  to construct $D$ from $\maxexcited{(v,w)}$, where ${\sf b}'\sim{\sf b}$. Below is an example for clarity.

\begin{example}\label{ex:identify}
In the top row is a sequence of diagrams $(D_0,\ldots,D_5)$, where the leftmost is $D_0=\maxexcited{(v,w)}$ for $v,w\in\avoid{8}$\footnote{Here $v=(6,7,8,1,2,3,4,5)$, $w=(1,2,6,3,4,7,5,8)$.}. Below are the corresponding $\Psi(D_i)$ for $0\leq i\leq 5$.
      \[\begin{picture}(440,100)
\put(0,25){\ytableausetup
{boxsize=0.9em}
{\begin{ytableau}
       \ &  \ &  0 &  0 &   0  \\
       \ & \ & \ &  \ &  0  \\
       \ & \ & \ &  \ &  \  \\
\end{ytableau}}}
\put(75,25){\ytableausetup
{boxsize=0.9em}
{\begin{ytableau}
       \ &  \ & 0 &  0 &  0  \\
       \ & \ & \ &  \ &  0  \\
       \ & \ & \ &  \ &  \  \\
\end{ytableau}}}
\put(150,25){\ytableausetup
{boxsize=0.9em}
{\begin{ytableau}
      \ &  \ & 0 &  0 &  0  \\
       \ & \ & \ &  \ &  0  \\
       \ & \ & \ &  \ &  \  \\
\end{ytableau}}}
\put(225,25){\ytableausetup
{boxsize=0.9em}
{\begin{ytableau}
       \ &  \ & 0 &  1 &  0  \\
       \ & \ & \ &  \ &  0  \\
       \ & \ & \ &  \ &  \  \\
\end{ytableau}}}
\put(300,25){\ytableausetup
{boxsize=0.9em}
{\begin{ytableau}
       \ &  \ & 0 &  1 &  0  \\
       \ & \ & \ &  \ &  1  \\
       \ & \ & \ &  \ &  \  \\
\end{ytableau}}}
\put(375,25){\ytableausetup
{boxsize=0.9em}
{\begin{ytableau}
       \ &  \ & 0 &  2 &  0  \\
       \ & \ & \ &  \ & 1  \\
       \ & \ & \ &  \ &  \  \\
\end{ytableau}}}
\put(0,80){\ytableausetup
{boxsize=0.9em}
{\begin{ytableau}
     \ &  \ & + &  + &  +  \\
      \ & \ & \ &  \ &  +  \\
       \ & \ & \ &  \ &  \  \\
\end{ytableau}}}
 \put(65,70){$\rightarrow$}
\put(75,80){\ytableausetup
{boxsize=0.9em}
{\begin{ytableau}
       \ &  \ & \ &  + &  +  \\
       \ & + & \ &  \ &  +  \\
      \ & \ & \ &  \ &  \  \\
\end{ytableau}}}
\put(140,70){$\rightarrow$}
\put(150,80){\ytableausetup
{boxsize=0.9em}
{\begin{ytableau}
      \ &  \ & \ &  + &  +  \\
       \ & \ & \ &  \ &  +  \\
      + & \ & \ &  \ &  \  \\
\end{ytableau}}}
\put(215,70){$\rightarrow$}
\put(225,80){\ytableausetup
{boxsize=0.9em}
{\begin{ytableau}
      \ &  \ & \ &  + &  +  \\
       \ & \ & \textcolor{blue}{\Plus} &  \ & +  \\
       + & \ & \ &  \ &  \  \\
\end{ytableau}}}
\put(290,70){$\rightarrow$}
\put(300,80){\ytableausetup
{boxsize=0.9em}
{\begin{ytableau}
     \ &  \ & \ &  + &  +  \\
       \ & \ & \textcolor{blue}{\Plus} &  \ & +  \\
       + & \ & \ &  \textcolor{blue}{\Plus} &  \  \\
\end{ytableau}}}
\put(365,70){$\rightarrow$}
\put(375,80){\ytableausetup
{boxsize=0.9em}
{\begin{ytableau}
      \ &  \ & \ &  + &  +  \\
       \ & \ & \textcolor{blue}{\Plus} &  \ &  + \\
       + & \textcolor{blue}{\Plus} & \ &  \textcolor{blue}{\Plus} &  \ \\
\end{ytableau}}}
\end{picture}
\]
These $\Psi(D_i)$ are drawn within $\codePD{v}$ for ease of reading.
\end{example}

For $U\in\dTab(\maxexcited{(v,w)})$, recall that:
\[\Pos{U}=\{{\sf b}\in \maxexcited{(v,w)} \, : \, U({\sf b})>0\}.\]
We now algorithmically construct a diagram $D_U\subseteq \codePD{v} $ from $U\in \dTab(\maxexcited{(v,w)})$.

\begin{construction}\label{const:org}
We will iteratively construct a diagram $D_U$ for $U\in \dTab(\maxexcited{(v,w)})$. If any ($K$-)excited move in this algorithm would cause the resulting diagram to leave $\codePD{v}$, we output that $D_U$ does not exist.

    First, partition $\Pos{U}$ into diagonals $\{{\sf C}_i\}_{i=0}^{\ell}$ for some $\ell\in\mathbb{Z}_{\geq 0}$ such that all boxes of ${\sf C}_i$ lie to the lower left of ${\sf C}_j$ if $i>j$.  
    Form this partition in order of increasing $i$, where ${\sf C}_i$ is of maximal length.
Initialize $D_U^{(0)}=\maxexcited{(v,w)}$ and ${\sf C}_i^{(0)}:={\sf C}_i$ for $0\leq i\leq \ell$.
For notation, we write ${\sf C}_j^{(i)}=\{{\sf c}_{j_k}^{(i)}\}\subseteq D_U^{(i)}$ for $k\in [\,|{\sf C}_j^{(i)}|\,]$, where $0\leq i\leq \ell+1$ and $0\leq j\leq \ell$.

    Let $S_i$ denote the set of boxes in $D_U^{(i)}-{\sf C}_i^{(i)}$ that lie to the bottom left of ${\sf C}_i^{(i)}$. 
    Consider each ${\sf b}\in S_i$, in order of left to right, and bottom to top. 
    Set 
    \begin{equation}\label{eq:numMovesShift}
        m=\max\Big(0,\{U({\sf c}_{i_k}^{(0)})-\gap({\sf b},{\sf c}_{i_k}^{(i)}) \, : \,
     {\sf c}_{i_k}^{(i)} \text{ lies to the upper right of } {\sf b}\}\Big).
    \end{equation}
    Then apply $m$ excited moves to ${\sf b}$. 
    Once these moves have been applied for each box in $S_i$, let $D_U^{(i+1)}$ be the resulting diagram. Let
    ${\sf C}_j^{(i+1)}=\{{\sf c}_{j_k}^{(i+1)}\}$ be the  boxes in $D_U^{(i+1)}$ such that ${\sf c}_{j_k}^{(i+1)}\sim {\sf c}_{j_k}^{(i)}$.
    We set $D_U^\circ:=D_U^{(\ell+1)}$.
    Now to $D_U^\circ$, apply $U({\sf c}_{i_k}^{(0)})$ many $K$-excited moves at each ${\sf c}_{i_k}^{(\ell+1)}\in {\sf C}_i^{(\ell+1)}$ for $0\leq i\leq \ell$. That is,  
    apply $K$-excited moves at 
    \[{\sf c}_{i_k}^{(\ell+1)}, {\sf c}_{i_k}^{(\ell+1)}+(1,-1),\ldots,{\sf c}_{i_k}^{(\ell+1)}+(U({\sf c}_{i_k}^{(0)})-1,-U({\sf c}_{i_k}^{(0)})+1)\]
    for each ${\sf c}_{i_k}^{(\ell+1)}\in {\sf C}_i^{(\ell+1)}$ where $0\leq i\leq \ell$.
    Call the resulting diagram $D_U$.
\end{construction}

\begin{example}\label{ex:showComput}
 Take $U\in \dTab(\maxexcited{(v,w)})$ as the tableau below, where $v,w$ are as in Example~\ref{ex:TvwConstr}. 
 \[\Small
 \ytableausetup
{boxsize=1em}
{
\begin{ytableau}
  \ &  0 &  0 &  0 \\
  \ &   1 &  0 &  0 \\
  \ &  \ &  0 &   0 &  0 &  0 \\
  \none &  \ &  \ &  1 &  1 &  0  \\
  \none &  \ &  \ &   0 &   1 &  0 &  0 &  0  \\
  \none &  \ &  \ &  \ &   0 &  1 &  1 &  0 &  \ &  1 \\
  \none &  \ &  \ &  \ &  \ &  0 &  0 &  2 &  \ &   0 \\
  \none &  \ &  \ &  \ &  \ &  \ &  \ &  \ &  \ &  \ &  1 &  0  \\
  \none &  \ &  \ &  \ &  \ &  \ &  \ &  \ &  \ &  \ &  0 &  1   \\
  \none &  \none &  \none &  \none &  \none &  \none &  \ &  \ &  \ &  \ &  \ &  \
\end{ytableau}}
\]
Below are the diagrams $\{D_U^{(i)}\}_{i=0}^3$. In each $D_U^{(i)}$, we have shaded the sets $\{{\sf C}_j^{(i)}\}$. Further, in $D_U^{(i)}$, we have bolded ${\sf C}_i^{(i)}$, for $0\leq i\leq 2$.
\[\Small\ytableausetup
{boxsize=0.8em}
{
\begin{ytableau}
  \ &  + &  + &  + \\
  \ &  *(gray!30)\Plus &  + &  + \\
  \ &  \ &  + &  + &  + &  + \\
  \none &  \ &  \ &  *(gray!30)+ &  *(gray!30)\Plus &  +  \\
  \none &  \ &  \ &  + &  *(gray!30)+ &  + &  + &  +  \\
  \none &  \ &  \ &  \ &  + &  *(gray!30)+ &  *(gray!30)+ &  + &  \ &  *(gray!30)\Plus \\
  \none &  \ &  \ &  \ &  \ &  + &  + &  *(gray!30)+ &  \ &  + \\
  \none &  \ &  \ &  \ &  \ &  \ &  \ &  \ &  \ &  \ &  *(gray!30)\Plus &  + \\
  \none &  \ &  \ &  \ &  \ &  \ &  \ &  \ &  \ &  \ &  + &  *(gray!30)\Plus  \\
  \none &  \none &  \none &  \none &  \none &  \none &  \ &  \ &  \ &  \ &  \ &  \
\end{ytableau}}
\ \
\raisebox{-1cm}{$\rightarrow$}
\
{
\begin{ytableau}
  \ &  + &  + &  + \\
  \ &  *(gray!30)+ &  + &  + \\
  \ &  \ &  + &  + &  + &  + \\
  \none &  \ &  \ &  \ &  *(gray!30)+ &  +  \\
  \none &  \ &  *(gray!30)\Plus &  \ &  \ &  + &  + &  +  \\
  \none &  \ &  + &  *(gray!30)\Plus &  \ &  *(gray!30)+ &  *(gray!30)\Plus &  + &  \ &  *(gray!30)+ \\
  \none &  \ &  \ &  + &  \ &  + &  + &  *(gray!30)\Plus &  \ &  \ \\
  \none &  \ &  \ &  \ &  \ &  \ &  \ &  \ &  + &  \ &  *(gray!30)+ &  + \\
  \none &  \ &  \ &  \ &  \ &  \ &  \ &  \ &  \ &  \ &  \ &  *(gray!30)+  \\
  \none &  \none &  \none &  \none &  \none &  \none &  \ &  \ &  \ &  + &  \ &  \
\end{ytableau}}
\ \
\raisebox{-1cm}{$\rightarrow$}
\
{
\begin{ytableau}
   \ &  + &  + &  + \\
  \ &  *(gray!30)+ &  + &  + \\
  \ &  \ &  + &  + &  + &  + \\
  \none &  \ &  \ &  \ &  *(gray!30)+ &  +  \\
  \none &  \ &  *(gray!30)+ &  \ &  \ &  + &  + &  +  \\
  \none &  \ &  \ &  *(gray!30)+ &  \ &  \ &  *(gray!30)+ &  + &  \ &  *(gray!30)+ \\
  \none &  + &  \ &  \ &  *(gray!30)\Plus &  \ &  \ &  *(gray!30)+ &  \ &  \ \\
  \none &  \ &  + &  \ &  \ &  \ &  \ &  \ &  + &  \ &  *(gray!30)+ &  + \\
  \none &  \ &  \ &  + &  + &  \ &  \ &  \ &  \ &  \ &  \ &  *(gray!30)+  \\
  \none &  \none &  \none &  \none &  \none &  \none &  \ &  \ &  \ &  + &  \ &  \
\end{ytableau}}
\ \
\raisebox{-1cm}{$\rightarrow$}
\
{
\begin{ytableau}
   \ &  + &  + &  + \\
  \ &  *(gray!30)+ &  + &  + \\
  \ &  \ &  + &  + &  + &  + \\
  \none &  \ &  \ &  \ &  *(gray!30)+ &  +  \\
  \none &  \ &  *(gray!30)+ &  \ &  \ &  + &  + &  +  \\
  \none &  \ &  \ &  *(gray!30)+ &  \ &  \ &  *(gray!30)+ &  + &  \ &  *(gray!30)+ \\
  \none &  + &  \ &  \ &  *(gray!30)+ &  \ &  \ &  *(gray!30)+ &  \ &  \ \\
  \none &  \ &  + &  \ &  \ &  \ &  \ &  \ &  + &  \ &  *(gray!30)+ &  + \\
  \none &  \ &  \ &  + &  + &  \ &  \ &  \ &  \ &  \ &  \ &  *(gray!30)+  \\
  \none &  \none &  \none &  \none &  \none &  \none &  \ &  \ &  \ &  + &  \ &  \
\end{ytableau}}
\]
Below, to the left is $D_U^{\circ}=D_U^{(3)}$. Then to the right is $D_U$. Here the plusses resulting from applying $K$-excited moves on $D_U^{\circ}$  are drawn in bolded blue.
\[\Small\ytableausetup
{boxsize=0.8em}
{\begin{ytableau}
  \ &  + &  + &  + \\
  \ &  + &  + &  + \\
  \ &  \ &  + &  + &  + &  + \\
  \none &  \ &  \ &  \ &  + &  +  \\
  \none &  \ &  + &  \ &  \ &  + &  + &  +  \\
  \none &  \ &  \ &  + &  \ &  \ &  + &  + &  \ &  + \\
  \none &  + &  \ &  \ &  + &  \ &  \ &  + &  \ &  \ \\
  \none &  \ &  + &  \ &  \ &  \ &  \ &  \ &  + &  \ &  + &  + \\
  \none &  \ &  \ &  + &  + &  \ &  \ &  \ &  \ &  \ &  \ &  +  \\
  \none &  \none &  \none &  \none &  \none &  \none &  \ &  \ &  \ &  + &  \ &  \
\end{ytableau}}
\qquad
\qquad
{\begin{ytableau}
  \ &  + &  + &  + \\
  \ &  + &  + &  + \\
  \textcolor{blue}{\Plus} &  \ &  + &  + &  + &  + \\
  \none &  \ &  \ &  \ &  + &  +  \\
  \none &  \ &  + &  \textcolor{blue}{\Plus} &  \ &  + &  + &  +  \\
  \none &  \textcolor{blue}{\Plus} &  \ &  + &  \ &  \ &  + &  + &  \ &  + \\
  \none &  + &  \textcolor{blue}{\Plus} &  \ &  + &  \textcolor{blue}{\Plus} &  \ &  + &  \textcolor{blue}{\Plus} &  \ \\
  \none &  \ &  + &  \textcolor{blue}{\Plus} &  \ &  \ &  \textcolor{blue}{\Plus} &  \ &  + &  \ &  + &  + \\
  \none &  \ &  \ &  + &  + &  \textcolor{blue}{\Plus} &  \ &  \ &  \ &  \textcolor{blue}{\Plus} &  \ &  +  \\
  \none &  \none &  \none &  \none &  \none &  \none &  \ &  \ &  \ &  + &  \textcolor{blue}{\Plus} &  \
\end{ytableau}}
\]
\vspace*{-0.3cm}
\end{example}

We say a subset $Z\subseteq \maxexcited{(v,w)}$ is a \mydef{zig-zag} if for any pair ${\sf{y}},{\sf{z}}\in Z$ and ${\sf{y}}',{\sf{z}}'\in D$ where $D\in\skewexcited{(v,w)}$ such that $\sf{y}\sim \sf{y}'$ and $\sf{z}\sim \sf{z}'$, we have that $\sf{y}'$ and $\sf{z}'$ do not form a diagonal. 
Order boxes in $Z=\{{\sf{z}}_i\}_{i\in[t]}$ such that indices increase from bottom left to top right and  say $Z$ \mydef{starts at} ${\sf{z}}_1$.

\begin{example} Below we have bolded different subsets $Z\subset \maxexcited{(v,w)}$, where $v,w$ are as in Example~\ref{ex:DtopCC}. The leftmost two are zig-zags, but the rightmost pair are not zig-zags.
 \[\Small\ytableausetup
{boxsize=0.9em}
{\begin{ytableau}
  \ &  \ &  + &  \Plus \\
 \ &  \ & + &  + &  \  &   + &  + &  + \\
 \ &  \ &  \Plus & + &  \  &  + & + &  + \\
 \ &  \ &  \ &  \ &  \  &  \ &  \ &  \ \\
 \none &  \ &  \ &  \ &  \  & + &  + &  + \\
  \none &  \none &  \none &    \ & \ &  \ &  \ &  \ 
\end{ytableau}}
\qquad
{\begin{ytableau}
  \ &  \ &  + &  + \\
 \ &  \ & + &  + &  \  &   \Plus &  + &  \Plus \\
 \ &  \ &  \Plus & \Plus &  \  &  + & + &  + \\
 \ &  \ &  \ &  \ &  \  &  \ &  \ &  \ \\
 \none &  \ &  \ &  \ &  \  & + &  + &  + \\
  \none &  \none &  \none &    \ & \ &  \ &  \ &  \ 
\end{ytableau}}
\qquad
{\begin{ytableau}
  \ &  \ &  + &  + \\
 \ &  \ & + &  \Plus &  \  &   \Plus &  + &  \Plus \\
 \ &  \ &  \Plus & + &  \  &  + & + &  + \\
 \ &  \ &  \ &  \ &  \  &  \ &  \ &  \ \\
 \none &  \ &  \ &  \ &  \  & + &  + &  + \\
  \none &  \none &  \none &    \ & \ &  \ &  \ &  \ 
\end{ytableau}}
\qquad
{\begin{ytableau}
  \ &  \ &  \Plus  & + \\
 \ &  \ & + &  + &  \  &   + &  + &  + \\
 \ &  \ &  + & \Plus  &  \  &  + & + &  + \\
 \ &  \ &  \ &  \ &  \  &  \ &  \ &  \ \\
 \none &  \ &  \ &  \ &  \  & + &  + &  + \\
  \none &  \none &  \none &    \ & \ &  \ &  \ &  \ 
\end{ytableau}}
\]
Both zig-zags in the left pair of diagrams start at $(3,3)$.
The bolded subset in the third is not a zig-zag since boxes $(2,4)$ and $(2,6)$ in $Z$ do not satisfy the definition. Similarly, in the fourth diagram, we find boxes $(1,3)$ and $(3,4)$ in $Z$ do not satisfy the definition.
\end{example}

\begin{proposition}\label{claim:preImg}
        Suppose $v\geq w$, where $v,w\in\avoid{n}$. Consider a tableau $U\in\dTab(\maxexcited{(v,w)})$. Then $D_U\in\kskewexcited{(v,w)}$ if and only if for each ${\sf b}\in \maxexcited{(v,w)}$ and each zig-zag $Z=\{{\sf{z}}_i\}_{i\in[t]}$ in  $\maxexcited{(v,w)}$  for some $t\in\mathbb{Z}_{>0}$ that starts at ${\sf b}={\sf{z}}_1$, we have:
        \begin{equation}\label{eq:preImg}
            \Big(\sum_{{\sf{z}}_i \in Z}U({\sf{z}}_i)\Big)-\Big(\sum_{i\in[t-1]}\gap({\sf{z}}_i,{\sf{z}}_{i+1})\Big)\leq \Tab{v,w}({\sf b}).
        \end{equation}
\end{proposition}
\begin{proof}
 By the order of moves in Construction~\ref{const:org}, the result follows once we confirm that $D_U$ exists. 
Note that by construction of $\Tab{v,w}$, we see $D_U\subseteq\codePD{v}$ if and only if no more than $\Tab{v,w}({\sf b})$ total moves are applied to boxes ${\sf b}'$ where ${\sf b}'\sim{\sf b}$. 

Consider a zig-zag $Z=\{{\sf{z}}_i\}_{i\in[t]}$ that starts at ${\sf b}={\sf{z}}_1$. If  $D_U\in\kskewexcited{(v,w)}$, then this means $U\in\Psi(\kskewexcited{(v,w)})$. Then at least 
\begin{equation}\label{eq:nogapZ}
    \Big(\sum_{{\sf{z}}_i \in Z}U({\sf{z}}_i)\Big)-\Big(\sum_{i\in[t-1]}\gap({\sf{z}}_i,{\sf{z}}_{i+1})\Big)
\end{equation}
 moves will be applied to  boxes ${\sf b}'\sim{\sf b}$ in Construction~\ref{const:org}. Thus 
   since $D_U\subseteq\codePD{v}$,  Equation~\eqref{eq:preImg} holds.

Similarly by Equation~\eqref{eq:numMovesShift}, the number of moves applied to ${\sf b}$ in Construction~\ref{const:org} is the maximal possible value of Equation~\eqref{eq:nogapZ} over all zig-zags $Z=\{{\sf{z}}_i\}_{i\in[t]}$ starting at ${\sf b}$. Thus if Equation~\eqref{eq:preImg} holds for each $Z$, then it follows that in Construction~\ref{const:org}, no more than  $\Tab{v,w}({\sf b})$ moves are applied to ${\sf b}$. Then the result follows.
\end{proof}

\begin{remark}\label{rem:org}
By Proposition~\ref{claim:preImg}, if $U=\Psi(D)$ for some $D\in\kskewexcited{(v,w)}$, then we have $D_U\in \kskewexcited{(v,w)}$. In particular, Construction~\ref{const:org} applies a minimal number of excited moves to $\maxexcited{(v,w)}$ such that $\Pos{(\Psi(D_U))}=\Pos{U}$. Further, $|D|=\sum_{{\sf b}\in \Pos{U}}U({\sf b})+{\sf inv}(w)$, so $|D|=|D_U|$.
\end{remark}

Let $D_{\tt zip}(v,w):=D_U^\circ$ and $D_{\tt zip}^K(v,w):=D_U$ for $U={\sf F}_{v,w}$.

\begin{example}\label{ex:showComputCont}
 Take $v,w$ as in Example~\ref{ex:TvwConstr}. Then for $U={\sf F}_{v,w}$, we 
 compute the diagrams $D_U^\circ $ and $D_U$ below.
 \[\Small\ytableausetup
{boxsize=0.9em}
{\begin{ytableau}
  \ &  + &  + &  + \\
  \ & \ &  + &  + \\
  + &  \ &  \ &  + &  + &  + \\
  \none &  + &  \ &  \ &  + &  +  \\
  \none &  \ &  + &  \ &  \ &  + &  + &  +  \\
  \none &  \ &  \ &  + &  \ &  \ &  + &  + &  \ &  + \\
  \none &  + &  \ &  \ &  + &  \ &  \ &  + &  \ &  + \\
  \none &  \ &  + &  \ &  \ &  + &  \ &  \ &  \ &  \ &  + &  + \\
  \none &  \ &  \ &  + &  \ &  \ &  \ &  \ &  \ &  \ &  \ &  +  \\
  \none &  \none &  \none &  \none &  \none &  \none &  \ &  \ &  \ &  + &  \ &  \
\end{ytableau}}
\qquad
\qquad
{\begin{ytableau}
  \ &  + &  + &  + \\
  \textcolor{blue}{\Plus} & \ &  + &  + \\
  + &  \textcolor{blue}{\Plus} &  \ &  + &  + &  + \\
  \none &  + &  \textcolor{blue}{\Plus} &  \ &  + &  +  \\
  \none &  \ &  + &  \textcolor{blue}{\Plus} &  \ &  + &  + &  +  \\
  \none &  \textcolor{blue}{\Plus} &  \ &  + &  \textcolor{blue}{\Plus} &  \ &  + &  + &  \ &  + \\
  \none &  + &  \textcolor{blue}{\Plus} &  \ &  + &  \textcolor{blue}{\Plus} &  \ &  + &  \ &  + \\
  \none &  \ &  + &  \textcolor{blue}{\Plus} &  \ &  + &  \textcolor{blue}{\Plus} &  \ &  \textcolor{blue}{\Plus} &  \ &  + &  + \\
  \none &  \textcolor{blue}{\Plus} &  \ &  + &  \textcolor{blue}{\Plus} &  \ &  \ &  \textcolor{blue}{\Plus} &  \ &  \textcolor{blue}{\Plus} &  \ &  +  \\
  \none &  \none &  \none &  \none &  \none &  \none &  \textcolor{blue}{\Plus} &  \ &  \ &  + &  \textcolor{blue}{\Plus} &  \
\end{ytableau}}
\]
Since $U={\sf F}_{v,w}$, these are $D_{\tt zip}(v,w)$ and $D_{\tt zip}^K(v,w)$, respectively. 
\end{example}

\begin{proposition}\label{claim:fav}
    Suppose $v\geq w$, where $v,w\in\avoid{n}$.
    Then $D_{\tt zip}(v,w)\in \skewexcited{(v,w)}$ and $D_{\tt zip}^K(v,w)\in \kskewexcited{(v,w)}$.
\end{proposition}
\begin{proof}
By their definition, any zig-zag in $\maxexcited{(v,w)}$ contains at most one box from each $\Lambda^{(i)}$ for $0\leq i\leq \ell$ and at most one box in $Q^{(0)}$. The construction of $\fTab{v,w}({\sf b})$ precisely subtracts values of $\fTab{v,w}({\sf d})$ from $\Tab{v,w}({\sf b})$, where ${\sf d}\in \bigcup_{i=0}^\ell\Lambda^{(i)}\cup Q^{(0)}$ lies in a zig-zag with ${\sf b}$ and lies to the upper right of ${\sf b}$.
Thus by Construction~\ref{constr:tabTi} and Proposition~\ref{claim:preImg}, 
the result follows. 
\end{proof}

Recall that 
$\Psi:\kskewexcited{(v,w)}\rightarrow\dTab(\maxexcited{(v,w)})$ is defined
 such that for $D\in\kskewexcited{(v,w)}$, the box ${\sf b}$ in $\Psi(D)$ is
 filled with the length of the chain of $K$-excited moves emanating from ${\sf b}'\sim{\sf b}$ 
  to construct $D$ from $\maxexcited{(v,w)}$.

\begin{proof}[Proof of Theorem~\ref{thm:321321IndDeg}]
By construction, $|D_{\tt zip}^K(v,w)|-{\sf inv}(w)=\CM{(v,w)}$. By Proposition~\ref{claim:fav}, we know
$D_{\tt zip}^K(v,w)\in \kskewexcited{(v,w)}$, so this proves the lower bound.

Now assume $v,w$ are {\UD}. Recall that in this case, $\mdCR{v,w}=\bigcup_{i=0}^\ell\Lambda^{(i)}$. Suppose $D\in\kskewexcited{(v,w)}$ is maximal. 
It remains to establish $|D|-{\sf inv}(w)\leq \CM{(v,w)}$.
Let $U=\Psi(D)$. By Remark~\ref{rem:org}, without loss of generality, we can assume  $D=D_U$.

Suppose $\Lambda^{(i)}=\{{\sf d}_{k,i}\}_{k\in[K_i]}$ where $|\Lambda^{(i)}|=K_i$ for $0\leq i\leq \ell$.
We partition $\maxexcited{(v,w)}$ into zig-zags $Z_k$, for $k\in[K_0]$.

\noindent{({\sf Step 0})}: Place ${\sf d}_{k,0}\in Z_k$ for $k\in[K_0]$.

For a diagram $D'$, let
\begin{align*}
    {\sf rows}(D')&=\{{\sf b}(1) \, : \, {\sf b}\in D'\}, \text{ and}\\
    {\sf cols}(D')&=\{{\sf b}(2)\, : \, {\sf b}\in D'\}.\\
\end{align*}
First, we partition boxes in the diagram 
\[\{{\sf b}\in\maxexcited{(v,w)} \, : \, {\sf b} \text{ lies strictly above }  \Lambda^{(0)}\}.\]

Consider $\{\Gamma_j\}$ the set of connected components of the set above. 
Then either $ {\sf rows}(\Gamma_j)\subseteq {\sf rows}(\Lambda^{(0)})$ or $ {\sf cols}(\Gamma_j)\subseteq {\sf cols}(\Lambda^{(0)})$.
We can never have both conditions satisfied simultaneously, or else $\Lambda^{(0)}$ was not maximal. 
If  $ {\sf rows}(\Gamma_j)\subseteq {\sf rows}(\Lambda^{(0)})$, place ${\sf b}\in \Gamma_j$ in $Z_{k}$ if 
${\sf d}_{k,0}(1)={\sf b}(1)$. Similarly if $ {\sf cols}(\Gamma_j)\subseteq {\sf cols}(\Lambda^{(0)})$, place ${\sf b}\in \Gamma_j$ in $Z_{k}$ if 
${\sf d}_{k,0}(2)={\sf b}(2)$.
Additionally, we see that if ${\sf b}\in \Gamma_j\cap Z_{k}$ and ${\sf b}'\in \Gamma_{j'}\cap Z_{k'}$ for $j\neq j'$, then 
$k\neq k'$. If not, the diagonal above $\Lambda^{(0)}$ must be longer than $\Lambda^{(0)}$, a contradiction.
Then each $Z_k$ at this stage is a zig-zag since the $\Gamma_j$ are connected components.

We now describe how to partition the boxes below $\Lambda^{(i-1)}$ for $i\geq 1$.

\noindent{({\sf Step i})}: For some $i\geq 1$, suppose the boxes weakly above $\Lambda^{(i-1)}$ have been placed into $\bigsqcup_{k\in[K_0]}Z_k$, where the $Z_k$ are all zig-zags at this stage. Consider the set 
\begin{align*}
    \Theta^i=\{{\sf b}\in\,&({\fTab{v,w}^{(i-1)}})_+ \, : \, {\sf b} \text{ weakly above }\Lambda^{(i)} \text{ and strictly below }\Lambda^{(i-1)}\}\\
&\cup \{{\sf b}\in \Lambda^{(i-1)} \, : \, \fTab{v,w}^{(i-1)}({\sf b})<\fTab{v,w}^{(i-2)}({\sf b})\}.
\end{align*}

\begin{adjustwidth}{0.25in}{0in}
\noindent{({\sf Part A})}: 
We place boxes into sets $A_k$, working one maximal length unit diagonal at a time.
Initialize these sets as singletons $A_k:=\{{\sf d}_{k,i}\}$ for $k\in [K_i]$. 
For this iterative algorithm, suppose the unit diagonal considered in the previous step is $X'\subseteq \maxexcited{(v,w)}$, where $X'\cap \bigsqcup_{k\in[K_i]}A_k=\{{\sf x}_{j}'\}_{j=1}^{K}$ for some $K\leq K_i$. For the initial step, set $X'$ to be the bottom rightmost unit diagonal in $\Lambda^{(i)}$.

Let $X\subseteq \maxexcited{(v,w)}$ be the boxes in the unit diagonal to the upper right of $X'$ such that each box in $X$ lies weakly to the lower left of $\Lambda^{(i-1)}$. 
Then from bottom right to upper left, find the boxes ${\sf x}_{j}\in X\cup \Lambda^{(i-1)}$ such that the following hold:
\[{\sf x}_{j}=\begin{cases}
    {\sf x}_{j}' & \text{ if } {\sf x}_{j}'\in \Lambda^{(i-1)}\\
    {\sf x}_{j}'+(0,1) & \text{ if }{\sf x}_{j}'+(0,1)\in \Theta^i \text{ and } {\sf x}_{j}'+(0,1)\neq {\sf x}_{j-1}\\
    {\sf x}_{j}'-(1,0) & \text{ otherwise.}\\
\end{cases}\]

Append ${\sf x}_{j}$ to $ A_s$ if its corresponding ${\sf x}_{j}'\in A_s$. By construction $A_s$ remains a zig-zag after each addition. The algorithm terminates when $X\subseteq \Lambda^{(i-1)}$. Let ${\mathcal A}=\cup_{s}A_s$. For $s\in [K_i]$, append $A_s$ to $Z_k$ if $A_s\cap Z_k\neq \emptyset$. By construction, $Z_k$ remains a zig-zag for each $k\in[K_0]$.
\end{adjustwidth}

Note that $\Lambda^{(i-1)}$ is defined as a bottom leftmost maximal length diagonal and $\Lambda^{(i)}$ lies to its lower left. Even more, we see each box in $\Lambda^{(i)}$ lies below a box in $\Lambda^{(i-1)}$. Thus the boxes ${\sf x}_j$ above exist. We pause to confirm $\Theta^i\subseteq{\mathcal A}$. 
To start, we note:
\begin{claim*} If ${\sf z}=(a,b)\in \Theta^i$ and ${\sf x}=(a',b')\in \maxexcited{(v,w)}$  is weakly above $\Lambda^{(i)}$ such that $a'\geq a$ and $b'\leq b$, then we have ${\sf x}\in \Theta^i$.
\end{claim*}
\begin{proof}
  This follows by the fact that below $\Lambda^{(i-1)}$ the tableau $\fTab{v,w}^{(i-1)}$ is weakly increasing down columns and weakly decreasing across rows.
\end{proof}

 We now show $\Theta^i\subseteq{\mathcal A}$. Suppose not, so there exists ${\sf x}\in \Theta^i-\mathcal{A}$. Take the bottommost then leftmost such ${\sf x}$. 
Say ${\sf x}$ lies in the unit diagonal $X$ considered in  {\sf Part A}, where the previous unit diagonal considered was $X'$. By the Claim, if ${\sf x}-(0,1)\in X'$, then ${\sf x}-(0,1)\in \Theta^i$. By the choice of ${\sf x}$ as maximally lower left, this implies ${\sf x}-(0,1)\in \mathcal{A}$. Analogously, we find that if  ${\sf x}+(1,0)\in X'$, then ${\sf x}+(1,0)\in \mathcal{A}$.

 First suppose  ${\sf x}-(0,1)\in X'$. Then by the previous paragraph, ${\sf x}-(0,1)\in \Theta^i\cap \mathcal{A}$. Then {\sf Part A} would have forced ${\sf x}\in \mathcal{A}$, since ${\sf x}$ would be placed in the same zig-zag as ${\sf x}+(1,0)$ or ${\sf x}-(0,1)$ by construction, a contradiction. Thus we assume ${\sf x}-(0,1)\not \in X'$.

  Since ${\sf x}\not\in \mathcal{A}$ we know ${\sf x}\not\in \Lambda^{(i)}$. This ensures at least one of ${\sf x}-(0,1)$ and ${\sf x}+(1,0)$ lies in $X'$. Since ${\sf x}-(0,1)\not \in X'$, we know ${\sf x}$ is leftmost in its row and ${\sf x}+(1,0)\in X'$.
  Among the boxes in $X'$ to the lower right of ${\sf x}$, 
  consider the leftmost boxes $Y=\{{\sf y}_j\}_{j=1}^J\subseteq X'$ such that
  \begin{itemize}
      \item {\sf Part A} places ${\sf y}_j$ and ${\sf y}_j+(0,1)$ in the same zig-zag, and
      \item  ${\sf y}_{j+1}={\sf y}_{j}+(1,1)$ for $j\in[J-1]$,
  \end{itemize}
  where the indices increase from left to right.
  Then it follows that ${\sf y}_j-(1,0)\in X$ for $j\in[J]$. 
  But by {\sf Part A} and the definition of $Y$, this implies that ${\sf y}_j+(0,1)\in \Theta^i$ for $j\in[J]$.
  By maximality of ${\Lambda}^{(i)}$, there must be weakly fewer boxes in $\{{\sf x}\}\cup\{{\sf y}_j+(0,1)\}_{j\in [J]}$ compared to the number of boxes in ${\Lambda}^{(i)}$ lying to its lower left. 
However by the fact that ${\sf x}-(0,1)\not \in X'$ and by the definition of $Y$, there are strictly more boxes in  $\{{\sf x}\}\cup\{{\sf y}_j+(0,1)\}_{j\in [J]}$ than  boxes in ${\Lambda}^{(i)}$ to its lower left,
a contradiction.
   Thus, $\Theta^i\subseteq {\mathcal{A}}$. 

We now proceed to the next part of the algorithm.  
\begin{adjustwidth}{0.25in}{0in}
\noindent{({\sf Part B})}: 
We now partition the boxes ${\sf b}\not\in {\mathcal A}$ that lie strictly above $\Lambda^{(i)}$ and weakly below $\Lambda^{(i-1)}$ into zig-zags $Z_s^B$. Initialize these sets as singletons $Z_{s}^B:=\{{\sf d}_{k_s,i-1}\}$, for $s\in [K_{i-1}-K_{i}]$, where 
$\{{\sf d}_{k_s,i-1}\}_{s\in [K_{i-1}-K_{i}]}=\Lambda^{(i-1)}-{\mathcal A}$.
We again proceed one unit diagonal at a time, but now from the upper right to the lower left. 
Suppose the most recent diagonal considered is $X'=\{{\sf x}_{j}'\}_{j=1}^{t}$, where $0\leq t\leq K_{i-1}-K_i$. Let $X$ be the next maximal length unit diagonal to the left of $X'$ with no boxes in ${\mathcal A}$ such that $X$ lies strictly above $\Lambda^{(i)}$. 
Then isolate boxes ${\sf x}_{j}\in X-{\mathcal A}$  such that the following hold:
\begin{itemize}
    \item 
    ${\sf x}_{j}\in\{{\sf x}_{j}'+(1,0),{\sf x}_{j}'-(0,1)\}$ and
    \item ${\sf x}_{j}$ are distinct.                                        \end{itemize}
Then append ${\sf x}_{j}$ to $Z_{s}^B$ if the corresponding ${\sf x}_{j}'\in Z_{s}^B$. Thus $Z_{s}^B$ remains a zig-zag by construction after each addition. The algorithm terminates once $X$ is empty. Then for $s\in [K_{i-1}-K_i]$, append $Z_{s}^B$ to $Z_k$ if $Z_{s}^B\cap Z_k\neq\emptyset $. By construction, $Z_k$ is a zig-zag for each $k\in[K_0]$.
\end{adjustwidth}
We claim that after {\sf Part B} has concluded, 
the region between $\Lambda^{(i-1)}$ and $\Lambda^{(i)}$ has been partitioned across the $Z_k$'s. Suppose not. Then some ${\sf x}\not\in{\mathcal A}$ does not appear in any $Z_k$.
Consider the top rightmost such ${\sf x}$. Since the length of the diagonals $X$ weakly decreases as {\sf Part B} continues, ${\sf x}$ can only exist if 
\[\{{\sf x}-(1,0),{\sf x}+(0,1)\}\subset {\mathcal A}.\] Note that since $\Theta^i\subseteq \mathcal{A}$, we know ${\sf x}\not\in \Theta^i$. 

By the Claim, if ${\sf x}-(1,0)\in \Theta^i$ or 
${\sf x}+(0,1)\in \Theta^i$, then ${\sf x}\in \Theta^i$, a contradiction. Thus we assume
$\{{\sf x}-(1,0),{\sf x}+(0,1)\}\subset {\mathcal A}-\Theta^i$.
Suppose ${\sf x}-(1,0)\in A_k\subseteq \mathcal{A}$. Then since ${\sf x}\not\in A_k$, this implies that  ${\sf x}-(1,1)\in A_k$. Note that the {\sf Part A} construction would pair ${\sf x}-(1,1)$ to ${\sf x}-(1,0)$ (rather than ${\sf x}-(2,1)$) only when ${\sf x}-(1,0)\in \Theta^i$. But just as above, ${\sf x}-(1,0)\in \Theta^i$ gives ${\sf x}\in \Theta^i$, a contradiction. 
Thus we have a partition of $\maxexcited{(v,w)}$ into zig-zags.

For simplicity of notation, reindex boxes ${\sf d}_{k,i}$ such that 
$\{{\sf d}_{k,i}\}=\Lambda^{(i)}\cap Z_k$ (when such a box exists). 
For $k\in[K_0]$, let $Z_k=\{{\sf z}_{k_j}\}_{j=1}^{J_k}$, where the labels in the lower subscript increase from  bottom left to top right, for some $J_k\in\mathbb{Z}_{> 0}$. That is, $Z_k$ starts at ${\sf z}_{k_1}$ for each $k\in[K_0]$.

Since  $U=\Psi(D)$,
we know      
\begin{equation}\label{eq:redistPos}
    \sum_{j=1}^{J_k}U({\sf z}_{k_j})\leq \Tab{v,w}({\sf z}_{k_1})=
\sum_{i=0}^\ell \fTab{v,w}({\sf d}_{k,i})
\end{equation}
for $k\in[K_0]$.
The inequality follows by Proposition~\ref{claim:preImg}. Note that if ${\sf z}_{k_1}$ lies below $\Lambda^{(\ell)}$,  $ \fTab{v,w}^{(\ell-1)}({\sf z}_{k_1})= \fTab{v,w}({\sf d}_{k,\ell})$ since $D^{(\ell)}=\emptyset$ by assumption. Then we see the equality follows by Construction~\ref{constr:tabTi}.

Now we aim to construct a new tableau $\Omega_3$ such that the following hold: 
\begin{itemize}
    \item $\Pos{(\Omega_3)}\subseteq \mdCR{v,w}$, 
    \item $D_{\Omega_3}\in\kskewexcited{(v,w)}$, and
    \item  $|D_{\Omega_3}|=|D|$. 
\end{itemize}
First define $\Omega_1\in \dTab(\maxexcited{(v,w)})$ such that $\Pos{(\Omega_1)}\subseteq \Lambda^{(0)}$, where
\[
\Omega_1({\sf d}_{k,0})=\sum_{j=1}^{J_k}U({\sf z}_{k_j})
\]
for $k\in[K_0]$.
We see $|D_{\Omega_1}|=|D|$ by construction.
Set $\Omega_2^{(0)}:=\Omega_1$.
Then for $0\leq i\leq\ell-1$,
define  $\Omega_2^{(i+1)}\in\dTab(\maxexcited{(v,w)})$ such that for ${\sf b}\in \maxexcited{(v,w)}$:
\[
\Omega_2^{(i+1)}({\sf b})=
\begin{cases}
  \fTab{v,w}({\sf d}_{k,i})& \text{ if }{\sf b}={\sf d}_{k,i} \text{ and } \Omega_2^{(i)}({\sf d}_{k,i})> \fTab{v,w}({\sf d}_{k,i}),\\
\Omega_2^{(i)}({\sf d}_{k,i})-\fTab{v,w}({\sf d}_{k,i}) & \text{ if }{\sf b}={\sf d}_{k,i+1} \text{ and } \Omega_2^{(i)}({\sf d}_{k,i})> \fTab{v,w}({\sf d}_{k,i}), \\
\Omega_2^{(i)}({\sf b}) & \text{ else}.       
\end{cases}
\]
Note that  $\Pos{(\Omega_2^{(i)})}\subseteq \bigcup_{j=0}^i\Lambda^{(j)}$ and that $|D_{\Omega_2^{(i)}}|=|D|$ by construction.
Let $\Omega_3=\Omega_2^{(\ell)}$.
By Equation~\eqref{eq:redistPos}, it follows that $\Omega_3({\sf d}_{k,\ell})\leq \fTab{v,w}({\sf d}_{k,\ell})$.
Thus $\Omega_3({\sf d}_{k,i})\leq \fTab{v,w}({\sf d}_{k,i})$ for any ${\sf d}_{k,i}\in \mdCR{v,w}$ by construction of $\Omega_3$.  This shows Equation~\eqref{eq:preImg} holds for $\Omega_3$, so by Proposition~\ref{claim:preImg}, $D_{\Omega_3}\in \kskewexcited{(v,w)}$.
In summary, 
\[|D|-{\sf inv}(w)=\sum_{{\sf d}_{k,i}\in \mdCR{v,w}}\Omega_3({\sf d}_{k,i})\leq \sum_{{\sf d}_{k,i}\in \mdCR{v,w}}\fTab{v,w}({\sf d}_{k,i})=\sum_{i=0}^\ell\sum_{{\sf d}\in \Lambda^{(i)}}\fTab{v,w}({\sf d})=\CM{(v,w)}.\]
The first equality above follows since $|D_{\Omega_3}|=|D_{\Omega_2^{(\ell)}}|=|D|$. The inequality above follows by construction of $\Omega_3$ via $\Omega_2^{(i)}$. The remaining equalities follow by reindexing and the definition of $\CM{(v,w)}$.
Then the result is proven.
\end{proof}

\subsection{Proof of Theorem~\ref{thm:321321IndReg} and Corollary~\ref{cor:321321IndAInv}}\label{sec:321321reg}
Using the results of the previous subsection, we can prove our main result.

\noindent
\emph{Proof of Theorem~\ref{thm:321321IndReg}:}
  This follows from Corollary~\ref{cor:unspGrSEYD}, Proposition~\ref{prop:pipeKLreg}, and Theorem~\ref{thm:321321IndDeg}.\qed

\noindent
\emph{Proof of Corollary~\ref{cor:321321IndAInv}:}
    This follows from Corollary~\ref{cor:unspGrSEYD}, Corollary~\ref{cor:pipeKLainv}, and Theorem~\ref{thm:321321IndDeg}.\qed

\begin{example}\label{ex:showThm}
 In Example~\ref{ex:TvwConstr}, we computed $\CM{(v,w)}=17$. As noted in that example, $v,w$ are {\UD}. Thus
Theorem~\ref{thm:321321IndDeg} determines \[\deg( \mathfrak G_{v,w}(\mathbf t))=|D_{\tt zip}^K(v,w)|=17+31=48.\]
Theorem~\ref{thm:321321IndReg} and  Corollary~\ref{cor:321321IndAInv} imply
\begin{align*}
    \reg (\mathbb{C}[{\bf z}^v]/J_{v,w})&=|D_{\tt zip}^K(v,w)|-{\sf inv}(w)=48-31=17, \text{ and}\\
    a(\mathbb{C}[{\bf z}^v]/J_{v,w})&=|D_{\tt zip}^K(v,w)|-{\sf inv}(v)=48-72=-24.
\end{align*}
We see $\reg (\mathbb{C}[{\bf z}^v]/J_{v,w})$ counts the bold blue plusses in $D_{\tt zip}^K(v,w)$ and $|a(\mathbb{C}[{\bf z}^v]/J_{v,w})|$ counts the empty boxes in $D_{\tt zip}^K(v,w)$.
\end{example}

We now consider an example in the non-{\UD} setting:
\begin{example}\label{ex:lowerBdEx}
 We return to $v,w$ as in Example~\ref{ex:DtopCC}. 
 We draw $\Tab{v,w}$ within $\codePD{v}$ on the left. 
 The second diagram from the left is  ${\sf F}_{v,w}$.
 The third is $D_{\tt zip}^K(v,w)$. However, this is not maximal. We find such a maximal diagram to the right. In both of these diagrams, the bold blue plusses are the plusses arising from $K$-excited moves.
 \[\begin{picture}(500,80)
\put(0,65){\Small\ytableausetup
{boxsize=1em}
{\begin{ytableau}
 \ &  \ &  1 &  1 \\
 \ &  \ & 1 &  1 &  \  &  2 &  2 &  2 \\
 \ &  \ &  1 &  1 &  \  &  2 &  2 &  2 \\
 \ &  \ &  \ &  \ &  \  &  \ &  \ &  \ \\
 \none &  \ &  \ &  \ &  \  &  1 &  1 &  1 \\
  \none &  \none &  \none &    \ & \ &  \ &  \ &  \ 
\end{ytableau}}}
\put(120,65){\Small\ytableausetup
{boxsize=1em}
{\begin{ytableau}
 \ &  \ &  0 &  0 \\
 \ &  \ & {1} &  0 &  \  &  {1} &  0 &  0 \\
 \ &  \ &  0 &  {1} &  \  &  0 &  {1} &  0 \\
 \ &  \ &  \ &  \ &  \  &  \ &  \ &  \ \\
 \none &  \ &  \ &  \ &  \  &  {1} &  0 &  0 \\
  \none &  \none &  \none &    \ & \ &  \ &  \ &  \ 
\end{ytableau}}}
\put(240,65){\Small\ytableausetup
{boxsize=1em}
{\begin{ytableau}
 \ &  \ &  + &  + \\
 \ &  \ &  + &  + &  \  &  + &  + &  + \\
 \ &  \textcolor{blue}{\Plus} &  \ &   + &  \textcolor{blue}{\Plus}  &  \ &  + &  + \\
 \ &  + &  \textcolor{blue}{\Plus} &  \ &  +  &  \textcolor{blue}{\Plus} &  \ &  \ \\
 \none &  \ &  \ &  \ &  \  &  + &  + & + \\
  \none &  \none &  \none &    \ & \textcolor{blue}{\Plus} &  \ &  \ &  \ 
\end{ytableau}}}
\put(360,65){\Small\ytableausetup
{boxsize=1em}
{\begin{ytableau}
   \ &  \ &  + &  + \\
 \ &  \textcolor{blue}{\Plus} &  \ &  + &  \  &  + &  + &  + \\
 \ &  + &  \textcolor{blue}{\Plus} &  \ &  \textcolor{blue}{\Plus}  &  \ &  + &  + \\
 \ &  + &  + &  \textcolor{blue}{\Plus} &  \  &  \textcolor{blue}{\Plus} &  \ &  \ \\
 \none &  \ &  \ &  + &  \textcolor{blue}{\Plus}  &  \ &  \ &  + \\
  \none &  \none &  \none &    \ & + &  + &  \textcolor{blue}{\Plus} &  \  
\end{ytableau}}}
\end{picture}
\]
Applying Corollary~\ref{cor:unspGrSEYD} and
Theorem~\ref{thm:321321IndDeg} to these computations, we find \[\deg( \mathfrak G_{v,w}(\mathbf t))=22> |D_{\tt zip}^K(v,w)|=20.\]
Theorem~\ref{thm:321321IndReg} and  Corollary~\ref{cor:321321IndAInv} compute
\begin{align*}
    \reg (\mathbb{C}[{\bf z}^v]/J_{v,w})&=7=22-15> |D_{\tt zip}^K(v,w)|-{\sf inv}(w)=5, \text{ and}\\
    a(\mathbb{C}[{\bf z}^v]/J_{v,w})&=-18=22-40> |D_{\tt zip}^K(v,w)|-{\sf inv}(v)=-20.
\end{align*}
Thus, the algorithmic lower bound given by $D_{\tt zip}^K(v,w)$ can be improved in this case.
\end{example}

\begin{remark}
   A previous version of this manuscript claimed an algorithm computing regularities for all $v,w\in \avoid{n}$. While that algorithm is not correct, we expect similar techniques can be used for a revised algorithm. 
\end{remark}

\section{Regularity of Ladder Determinantal Varieties}\label{sec:regLadd}
In this section we use the result of Escobar--Fink--Rajchgot--Woo \cite{EFRW} which states two-sided ladder determinantal varieties are Kazhdan--Lusztig varieties indexed by $v, w\in\avoid{n}$. In this setting, $v,w$ are particular {\UD} pairs. We give translations of Theorem~\ref{thm:321321IndReg} and Corollary~\ref{cor:321321IndAInv} accordingly.

Lastly in this two-sided ladder case, we reformulate Theorem~\ref{thm:321321IndReg} and Corollary~\ref{cor:321321IndAInv} in terms of lattice paths.
This generalizes work of Ghorpade--Krattenthaler  \cite{GK15}.

\subsection{Ladder Determinantal Varieties}
A \mydef{ladder region} $L$ is a skew Young diagram $\lambda/\mu$. Let $\ell(\lambda)$ denote the number of non-negative parts of $\lambda$. Without loss of generality
we assume $\ell(\lambda)=\ell(\mu)$.
For $L=\lambda/\mu$,
we define the \mydef{perimeter} of $L$ as $2n$, where $n=\lambda_1+\lambda'_1$, 
{\emph{i.e.}}, the number of boxes in the first row plus the number of boxes in the first column of $\lambda$.

A ladder region $L$ is equivalently determined by its lower left corners $L^{\sf SW} = \{\alpha_i\}_{i\in [s]}$ and upper right corners $L^{\sf{NE}} = \{\beta_i\}_{i\in [t]}$, with points ordered upper left to bottom right. 
Define $\alpha_{0}=(0,0)$ to be the upper leftmost corner of $L$
and let $\alpha_{s+1}$ denote the bottom rightmost corner of $L$. For a point ${\gamma}$ in $L$ write $\gamma=(\gamma(1),\gamma(2))$.
A box ${\sf b}$ in $L$ inherits the label of its bottom right corner.

Let $\mathcal{M}=\{({\sf{p}}_i,r_i)\}_{i\in[s']}$ denote a set of marked points along the lower left border of $L$ where $r_i\in \mathbb{Z}_{> 0}$. Points in $\mathcal{M}$ are ordered from upper left to bottom right.
 Define $L(z)$ as the array of entries in positions of $L$ where position $(i,j)$ is filled with indeterminate $z_{ij}$. Take $\mathbb{C}[L(z)]$ the polynomial ring generated by entries in $L(z)$. 
Define the \mydef{two-sided mixed ladder determinantal ideal} $I_{L,\mathcal{M}}$:
\[I_{L,\mathcal{M}} := \langle r_i \text{\,-\,minors }  \text{ in } L_{[{\sf{p}}_i(1)],[{\sf{p}}_i(2)+1,\alpha_{s+1}(2)]}(z) \, : \, ({\sf{p}}_i,r_i)\in \mathcal{M}\rangle\subseteq \mathbb{C}[L(z)],\]
where $L_{I,J}(z)$ denotes the submatrix of $L(z)$ with row indices in $I$ and column indices in $J$ for $I,J\subseteq[n]$. 
 The \mydef{two-sided mixed ladder determinantal variety} has coordinate ring $X(L,\mathcal{M}):=\mathbb{C}[L(z)]/I_{L,\mathcal{M}}$.
Taking $L=\lambda$, \emph{i.e.}, when $\mu=\emptyset$, produces a \mydef{one-sided mixed ladder determinantal variety}. 
Define $(L,{\mathcal M})$ to be a \mydef{minimal ladder} if
\begin{enumerate}
    \item each $z_{ij}\in L(z)$ appears in a monomial of a generator in $I_{L,\mathcal{M}}$, 
    \item $ {\sf{p}}_1(1)-r_1<{\sf{p}}_2(1)-r_2<\dots<{\sf{p}}_{s'}(1)-r_{s'}$,
    \item $ {\sf{p}}_1(2)+r_1<{\sf{p}}_2(2)+r_2<\dots<{\sf{p}}_{s'}(2)+r_{s'}$, and
    \item  $({\sf{p}}_i(1)-r_i+1,{\sf{p}}_i(2)+r_i)\in L$ for each $i\in[s']$.
\end{enumerate}
It is straightforward to reduce any two-sided ladder to a minimal two-sided ladder.

\begin{example}\label{ex:2sidelad}
Let $L=\lambda/\mu$, where $\lambda=(5,5,5,5,2,2)$ and $\mu=(2,1,0,0,0,0)$. Then $L^{\sf SW}=\{(4,0),(6,3)\}$ and $L^{\sf NE}=\{(0,3),(1,4),(2,5)\}$. Below is $L(z)$ with marked points $\mathcal{M}=\{((4,0),3),((4,2),2),((6,3),2)\}$ drawn in red.

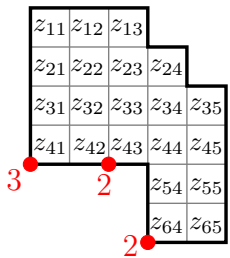
\[\begin{picture}(100,90)
\put(0,0){\begin{tikzpicture}[scale=.52]
\draw[line width = .1ex, gray] (0,5) -- (3,5);
\draw[line width = .1ex, gray] (0,4) -- (4,4);
\draw[line width = .1ex, gray] (0,3) -- (5,3);
\draw[line width = .1ex, gray] (3,2) -- (5,2);
\draw[line width = .1ex, gray] (3,1) -- (5,1);

\draw[line width = .1ex, gray] (1,6) -- (1,2);
\draw[line width = .1ex, gray] (2,6) -- (2,2);
\draw[line width = .1ex, gray] (3,5) -- (3,2);
\draw[line width = .1ex, gray] (4,4) -- (4,0);

\draw[line width = .25ex] (0,2)--(3,2)--(3,0)--(5,0)--(5,4)--(4,4)--(4,5)--(3,5)--(3,6)--(0,6)--(0,2);
\filldraw[red] (0,2) circle (1ex);
\filldraw[red] (2,2) circle (1ex);
\filldraw[red] (3,0) circle (1ex);
\put(1,80){{\Small$z_{11}$}}
\put(15,80){{\Small$z_{12}$}}
\put(30,80){{\Small$z_{13}$}}
\put(1,65){{\Small$z_{21}$}}
\put(15,65){{\Small$z_{22}$}}
\put(30,65){{\Small$z_{23}$}}
\put(45,65){{\Small$z_{24}$}}
\put(1,50){{\Small$z_{31}$}}
\put(15,50){{\Small$z_{32}$}}
\put(30,50){{\Small$z_{33}$}}
\put(45,50){{\Small$z_{34}$}}
\put(60,50){{\Small$z_{35}$}}
\put(1,35){{\Small$z_{41}$}}
\put(16,35){{\Small$z_{42}$}}
\put(30,35){{\Small$z_{43}$}}
\put(45,35){{\Small$z_{44}$}}
\put(60,35){{\Small$z_{45}$}}
\put(45,20){{\Small$z_{54}$}}
\put(60,20){{\Small$z_{55}$}}
\put(45,5){{\Small$z_{64}$}}
\put(60,5){{\Small$z_{65}$}}
\put(-9,20){{$\textcolor{red}3$}}
\put(25,18){{$\textcolor{red}2$}}
\put(35,-5){{$\textcolor{red}2$}}
\end{tikzpicture}}
\end{picture}
\]
 Then $ I_{L,\mathcal M}=\langle 3\text{\,-\,minors of } L_{[4],[5]}(z), 2\text{\,-\,minors of } L_{[4],\{3,4,5\}}(z) ,2\text{\,-\,minors of } L_{[6],\{4,5\}}(z)\rangle$.\hspace{-0.1cm}
\end{example}

\subsection{Two-sided ladders and Kazhdan--Lusztig Varieties}
Let $(L,{\mathcal M})$ be a minimal two-sided ladder where  $L=\lambda/\mu$, ${\mathcal M}=\{({\sf{p}}_i,r_i)\}_{i\in[s']}$, and $L$ has perimeter $2n$. 
Define $s_v\in\mathbb{Z}_{\geq 0}^n$ as the sequence
\[s_v:=(\lambda_1-\mu_1,0^{\lambda_1-\lambda_2},\lambda_2-\mu_2,0^{\lambda_2-\lambda_3},
\ldots,\lambda_{\ell(\lambda)}-\mu_{\ell(\lambda)},0^{\lambda_{\ell(\lambda)}}).\]
Let $v\in S_n$ be the unique permutation such that ${\sf code}(v)=s_v$.
Suppose 
$L^{\sf{NE}} = \{\beta_j\}_{j\in [t]}$. 
Then take $w\in S_n$ to be the minimal length permutation satisfying  \[{\sf rank}_w((\|{\sf{p}}_i\|,\|{\beta}_j\|))=\min\big(\{\|{\sf{p}}_i\|,\|{\beta}_j\|,{\sf rank}_v((\|{\sf{p}}_i\|,\|{\beta}_j\|))+r_i-1\}\big)\]
for each ${i\in[s'],j\in[t]}$.
Here $\|\gamma\|=\gamma(1)+\gamma(2)$ for a point $\gamma$.
We define $\perm(L,\mathcal{M})=(v,w)$. 
This formula to compute $\perm(L,\mathcal{M})$ refines the formula in \cite[Theorem~4.7.3]{GM} for the one-sided ladder case.

For $(v,w)=\perm(L,\mathcal{M})$ we have 
$\mathcal{R}_v=\lambda/\mu$. 
Thus we identify $\mathbb{C}[{\bf z}^v]$ with $\mathbb{C}[L(z)]$ through the map $z_{ij}\mapsto z_{i'j'}$
where $(i',j')=\phi_v(i,j)$.

\begin{example}\label{ex:permForLadd} Let $L$ and $\mathcal{M}$ be as in Example~\ref{ex:2sidelad}. Below are $D(v)$ and $D(w)$ such that $(v,w)=\perm(L,\mathcal{M})$. In $D(w)$, the positions $\{(\|{\sf{p}}_i\|,\|{\beta}_j\|)\}_{i\in[s'],\\ j\in[t]}$ are shaded.

 
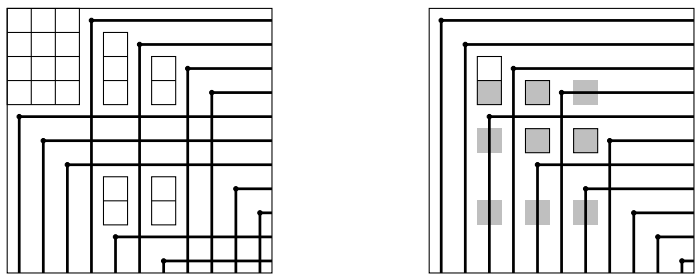
\[
\begin{tikzpicture}[scale=.32]
\draw (0,0) rectangle (11,11);

 \draw (0,11) rectangle (3,7);
 \draw (4,10) rectangle (5,7);
 \draw (6,9) rectangle (7,7);

\draw (4,4) rectangle (5,2);
 \draw (6,4) rectangle (7,2);
 
 \draw (0,10)--(3,10);
\draw (0,9)--(3,9);
 \draw (0,8)--(3,8);
\draw (1,11)--(1,7);
\draw (2,11)--(2,7);

\draw (5,9)--(4,9);
\draw (5,8)--(4,8);

\draw (7,8)--(6,8);

\draw (4,3)--(5,3);
\draw (6,3)--(7,3);

\filldraw (3.5,10.5) circle (.5ex);
\draw[line width = .2ex] (3.5,0) -- (3.5,10.5) -- (11,10.5);
\filldraw (5.5,9.5) circle (.5ex);
\draw[line width = .2ex] (5.5,0) -- (5.5,9.5) -- (11,9.5);
\filldraw (7.5,8.5) circle (.5ex);
\draw[line width = .2ex] (7.5,0) -- (7.5,8.5) -- (11,8.5);
\filldraw (8.5,7.5) circle (.5ex);
\draw[line width = .2ex] (8.5,0) -- (8.5,7.5) -- (11,7.5);
\filldraw (0.5,6.5) circle (.5ex);
\draw[line width = .2ex] (0.5,0) -- (0.5,6.5) -- (11,6.5);
\filldraw (1.5,5.5) circle (.5ex);
\draw[line width = .2ex] (1.5,0) -- (1.5,5.5) -- (11,5.5);
\filldraw (2.5,4.5) circle (.5ex);
\draw[line width = .2ex] (2.5,0) -- (2.5,4.5) -- (11,4.5);
\filldraw (9.5,3.5) circle (.5ex);
\draw[line width = .2ex] (9.5,0) -- (9.5,3.5) -- (11,3.5);
\filldraw (10.5,2.5) circle (.5ex);
\draw[line width = .2ex] (10.5,0) -- (10.5,2.5) -- (11,2.5);
\filldraw (4.5,1.5) circle (.5ex);
\draw[line width = .2ex] (4.5,0) -- (4.5,1.5) -- (11,1.5);
\filldraw (6.5,0.5) circle (.5ex);
\draw[line width = .2ex] (6.5,0) -- (6.5,0.5) -- (11,0.5);
\end{tikzpicture}
\hspace{5em} 
\begin{tikzpicture}[scale=.32]
\fill[lightgray]  (3,8) rectangle (2,7);
\fill[lightgray]  (5,8) rectangle (4,7);
\fill[lightgray]  (7,8) rectangle (6,7);

\fill[lightgray]  (3,6) rectangle (2,5);
\fill[lightgray]  (5,6) rectangle (4,5);
\fill[lightgray]  (7,6) rectangle (6,5);

\fill[lightgray]  (3,3) rectangle (2,2);
\fill[lightgray]  (5,3) rectangle (4,2);
\fill[lightgray]  (7,3) rectangle (6,2);

\draw (3,7) rectangle (2,9);
\draw (3,8)--(2,8);

 \draw (6,5) rectangle (7,6);
 
 \draw (4,5) rectangle (5,6);

 \draw (4,7) rectangle (5,8);

\draw (6,4) rectangle (7,2);
\draw (6,3)--(7,3);

\filldraw (0.5,10.5) circle (.5ex);
\draw[line width = .2ex] (0.5,0) -- (0.5,10.5) -- (11,10.5);
\filldraw (1.5,9.5) circle (.5ex);
\draw[line width = .2ex] (1.5,0) -- (1.5,9.5) -- (11,9.5);
\filldraw (3.5,8.5) circle (.5ex);
\draw[line width = .2ex] (3.5,0) -- (3.5,8.5) -- (11,8.5);
\filldraw (5.5,7.5) circle (.5ex);
\draw[line width = .2ex] (5.5,0) -- (5.5,7.5) -- (11,7.5);
\filldraw (2.5,6.5) circle (.5ex);
\draw[line width = .2ex] (2.5,0) -- (2.5,6.5) -- (11,6.5);
\filldraw (7.5,5.5) circle (.5ex);
\draw[line width = .2ex] (7.5,0) -- (7.5,5.5) -- (11,5.5);
\filldraw (4.5,4.5) circle (.5ex);
\draw[line width = .2ex] (4.5,0) -- (4.5,4.5) -- (11,4.5);
\filldraw (8.5,3.5) circle (.5ex);
\draw[line width = .2ex] (8.5,0) -- (8.5,3.5) -- (11,3.5);
\filldraw (9.5,2.5) circle (.5ex);
\draw[line width = .2ex] (9.5,0) -- (9.5,2.5) -- (11,2.5);
\filldraw (6.5,1.5) circle (.5ex);
\draw[line width = .2ex] (6.5,0) -- (6.5,1.5) -- (11,1.5);
\filldraw (10.5,0.5) circle (.5ex);
\draw[line width = .2ex] (10.5,0) -- (10.5,0.5) -- (11,0.5);
\draw (0,0) rectangle (11,11);
\end{tikzpicture}
\vspace*{-0.5cm}
\]
\end{example}

One-sided ladder determinantal varieties are isomorphic to vexillary matrix Schubert varieties, see \cite{GM,KMY}. In general, two-sided ladder determinantal varieties are not isomorphic to matrix Schubert varieties. For example, if $(L,{\mathcal M})$ is
as in Example~\ref{ex:2sidelad}, $X(L,{\mathcal M})$ is not isomorphic to a matrix Schubert variety.
As proven by Escobar--Fink--Rajchgot--Woo, all two-sided ladders can be realized as Kazhdan--Lusztig varieties:
\begin{theorem}\cite{EFRW}\label{prop:2sidedKL}
    Given $(L,{\mathcal M})$ minimal, suppose $\perm(L,\mathcal{M})=(v,w)$ and $L$ has perimeter $2n$. Then the following hold:
    \begin{enumerate}
        \item  $v,w\in \avoid{n}$ where $v\geq w$, and
        \item $I_{L,\mathcal{M}}$ and $J_{v,w}$ have the same set of generators.
    \end{enumerate}
\end{theorem}

\subsection{Specializing Theorem~\ref{thm:321321IndReg}}
When $(v,w)=\perm(L,\mathcal{M})$ for $(L,{\mathcal M})$ minimal, diagrams in $\kskewexcited(v,w)$ exhibit additional structure.  This allows us to re-frame Theorem~\ref{thm:321321IndReg} and Corollary~\ref{cor:321321IndAInv} in terms of lattice paths in $L$.

\begin{construction}[Computing boundary points]\label{alg:LaddPointCompute}
Take a minimal two-sided ladder $(L,{\mathcal M})$ where ${\mathcal M}=\{({\sf{p}}_i,r_i)\}_{i\in[s']}$ and $L^{\sf SW} = \{\alpha_i\}_{i\in [s]}$. 
By the first condition for minimality, note that we will have a marked point in each corner $\alpha_i$.
For each $i\in[s]$ let 
\begin{align*}
    r_i^H&:=\min\{r \, : \, \big((\alpha_i(1),y),r\big)\in  {\mathcal{M}}\}, \text{ and}\\
    r_{i}^V&:=\min\{r \, : \, \big((x,\alpha_i(2)),r\big)\in  {\mathcal{M}}\}.
\end{align*}
Initialize ${\mathcal M}'={\mathcal M}$.
For each $i\in[s-1]$, if 
$((\alpha_i(1),\alpha_{i+1}(2)),r)\not\in \mathcal{M}'$ for any $r\in\mathbb{Z}_{>0}$, append 
\[\big((\alpha_i(1),\alpha_{i+1}(2)),\min(r_i^H,r_{i+1}^V)\big)\]
to $\mathcal{M}'$. Lastly append $(\alpha_0,1)$ and $(\alpha_{s+1},1)$ to ${\mathcal{M}}'$. We have ${\mathcal{M'}}=\bigcup_{i\in[s]}{\mathcal{M}}^{V}_i\cup {\mathcal{M}}^{H}_i$, where
\begin{align*}
    {\mathcal{M}}^{V}_i&:=\{({\sf p},r) \in\mathcal{M}'\, : \, {\sf p}=(x,\alpha_i(2))\}, \text{ and}\\
   {\mathcal{M}}^{H}_i&:=\{({\sf p},r) \in\mathcal{M}' \, : \, {\sf p}=(\alpha_i(1),y)\}.
\end{align*}
Points in
${\mathcal{M}}^{V}_i$ are ordered top to bottom and those in  ${\mathcal{M}}^{H}_i$ are ordered right to left.

Initialize $V(\mathcal{M})=\emptyset$ and  $H(\mathcal{M})=\emptyset$. 
Then iterate the following for each $i\in [s]$:
\begin{itemize}[\textbf{-}]
  \item Write ${\mathcal{M}}^{V}_i=\{({\sf q}_j,t_j)\}_{j\in[|{\mathcal{M}}^{V}_i|]}$, ordered top to bottom. Then for 
    $j\in [|{\mathcal{M}}^{V}_i|-1]$, if $t_{j+1}-t_j=k\geq 1$, append ${\sf q}_j+(k'-\frac{1}{2},0)$ to $V(\mathcal{M})$ for each $k'\in[k]$. 
  \item Write ${\mathcal{M}}^{H}_i=\{({\sf q}_j,t_j)\}_{j\in[|{\mathcal{M}}^{H}_i|]}$, ordered right to left. Then for 
    $j\in [|{\mathcal{M}}^{H}_i|-1]$, if $t_{j+1}-t_j=k\geq 1$, append ${\sf q}_j+(0,-k'+\frac{1}{2})$ to $H(\mathcal{M})$ for each $k'\in[k]$.   
\end{itemize}
 This gives boundary points $V(\mathcal{M})$ and $H(\mathcal{M})$.

     Suppose $|V(\mathcal{M})|=N\in\mathbb{Z}_{\geq0}$. Then $|H(\mathcal{M})|=N$ by construction.
For $i\in[N]$, label points $H_i\in H(\mathcal{M})$ in increasing order from right to left.
For $i\in[N]$ in decreasing order, assign label $V_i$ to be the bottom-most point in $V(\mathcal{M})-\{V_j \, : \, N\geq j>i\}$ that lies to the upper left of $H_i$.
\end{construction}

\begin{example}\label{ex:laddAlgIll}
We illustrate Construction~\ref{alg:LaddPointCompute} below.
    The left diagram draws a ladder $L$ with the original marked points $\mathcal{M}$ bolded in red and $\mathcal{M}'-\mathcal{M}$ in light gray. The middle diagram adds $V(\mathcal{M})$ and $H(\mathcal{M})$ in bold black. The right diagram labels $V(\mathcal{M})$ and $H(\mathcal{M})$.

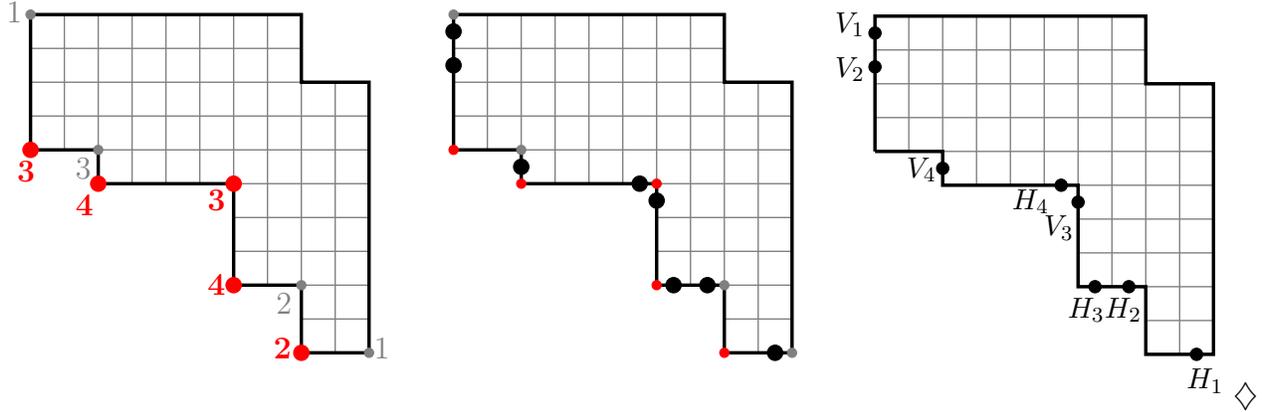
\[\begin{picture}(465,120)
\put(0,-5){\begin{tikzpicture}[scale=.45]
\draw[line width = .1ex, gray] (0,9) -- (8,9);
\draw[line width = .1ex, gray] (0,8) -- (8,8);
\draw[line width = .1ex, gray] (0,7) -- (10,7);
\draw[line width = .1ex, gray] (2,6) -- (10,6);
\draw[line width = .1ex, gray] (6,5) -- (10,5);
\draw[line width = .1ex, gray] (6,4) -- (10,4);
\draw[line width = .1ex, gray] (6,3) -- (10,3);
\draw[line width = .1ex, gray] (6,2) -- (10,2);
\draw[line width = .1ex, gray] (8,1) -- (10,1);

\draw[line width = .1ex, gray] (1,10) -- (1,6);
\draw[line width = .1ex, gray] (2,10) -- (2,6);
\draw[line width = .1ex, gray] (3,10) -- (3,5);
\draw[line width = .1ex, gray] (4,10) -- (4,5);
\draw[line width = .1ex, gray] (5,10) -- (5,5);
\draw[line width = .1ex, gray] (6,10) -- (6,5);
\draw[line width = .1ex, gray] (7,10) -- (7,2);
\draw[line width = .1ex, gray] (8,8) -- (8,2);
\draw[line width = .1ex, gray] (9,8) -- (9,0);

\draw[line width = .25ex] (0,6)--(0,10)--(8,10)--(8,8)--(10,8)--(10,0)--(8,0)--(8,2)--(6,2)--(6,5)--(2,5)--(2,6)--(0,6);
\put(-5,65){{$\textcolor{red}{\bf{3}}$}}
\put(17,52){{$\textcolor{red}{\bf{4}}$}}
\put(67,54){{$\textcolor{red}{\bf{3}}$}}
\put(67,22){{$\textcolor{red}{\bf{4}}$}}
\put(92,-2){{$\textcolor{red}{\bf{2}}$}}
\filldraw[red] (0,6) circle (1.25ex);
\filldraw[red] (2,5) circle (1.25ex);
\filldraw[red] (6,5) circle (1.25ex);
\filldraw[red] (6,2) circle (1.25ex);
\filldraw[red] (8,0) circle (1.25ex);
\put(-9,125){{$\textcolor{gray}1$}}
\put(17,66){{$\textcolor{gray}3$}}
\put(93,15){{$\textcolor{gray}2$}}
\put(130,-2){{$\textcolor{gray}1$}}
\filldraw[gray] (0,10) circle (0.75ex);
\filldraw[gray] (2,6) circle (0.75ex);
\filldraw[gray] (8,2) circle (0.75ex);
\filldraw[gray] (10,0) circle (0.75ex);
\end{tikzpicture}}
\put(160,-5){\begin{tikzpicture}[scale=.45]
\draw[line width = .1ex, gray] (0,9) -- (8,9);
\draw[line width = .1ex, gray] (0,8) -- (8,8);
\draw[line width = .1ex, gray] (0,7) -- (10,7);
\draw[line width = .1ex, gray] (2,6) -- (10,6);
\draw[line width = .1ex, gray] (6,5) -- (10,5);
\draw[line width = .1ex, gray] (6,4) -- (10,4);
\draw[line width = .1ex, gray] (6,3) -- (10,3);
\draw[line width = .1ex, gray] (6,2) -- (10,2);
\draw[line width = .1ex, gray] (8,1) -- (10,1);

\draw[line width = .1ex, gray] (1,10) -- (1,6);
\draw[line width = .1ex, gray] (2,10) -- (2,6);
\draw[line width = .1ex, gray] (3,10) -- (3,5);
\draw[line width = .1ex, gray] (4,10) -- (4,5);
\draw[line width = .1ex, gray] (5,10) -- (5,5);
\draw[line width = .1ex, gray] (6,10) -- (6,5);
\draw[line width = .1ex, gray] (7,10) -- (7,2);
\draw[line width = .1ex, gray] (8,8) -- (8,2);
\draw[line width = .1ex, gray] (9,8) -- (9,0);

\draw[line width = .25ex] (0,6)--(0,10)--(8,10)--(8,8)--(10,8)--(10,0)--(8,0)--(8,2)--(6,2)--(6,5)--(2,5)--(2,6)--(0,6);

\filldraw[red] (0,6) circle (0.75ex);
\filldraw[red] (2,5) circle (0.75ex);
\filldraw[red] (6,5) circle (0.75ex);
\filldraw[red] (6,2) circle (0.75ex);
\filldraw[red] (8,0) circle (0.75ex);
\filldraw[gray] (0,10) circle (0.75ex);
\filldraw[gray] (2,6) circle (0.75ex);
\filldraw[gray] (8,2) circle (0.75ex);
\filldraw[gray] (10,0) circle (0.75ex);

\filldraw[black] (0,9.5) circle (1.25ex);
\filldraw[black] (0,8.5) circle (1.25ex);
\filldraw[black] (2,5.5) circle (1.25ex);
\filldraw[black] (6,4.5) circle (1.25ex);

\filldraw[black] (5.5,5) circle (1.25ex);
\filldraw[black] (6.5,2) circle (1.25ex);
\filldraw[black] (7.5,2) circle (1.25ex);
\filldraw[black] (9.5,0) circle (1.25ex);
\end{tikzpicture}}
\put(320,-5){\begin{tikzpicture}[scale=.45]
\draw[line width = .1ex, gray] (0,9) -- (8,9);
\draw[line width = .1ex, gray] (0,8) -- (8,8);
\draw[line width = .1ex, gray] (0,7) -- (10,7);
\draw[line width = .1ex, gray] (2,6) -- (10,6);
\draw[line width = .1ex, gray] (6,5) -- (10,5);
\draw[line width = .1ex, gray] (6,4) -- (10,4);
\draw[line width = .1ex, gray] (6,3) -- (10,3);
\draw[line width = .1ex, gray] (6,2) -- (10,2);
\draw[line width = .1ex, gray] (8,1) -- (10,1);

\draw[line width = .1ex, gray] (1,10) -- (1,6);
\draw[line width = .1ex, gray] (2,10) -- (2,6);
\draw[line width = .1ex, gray] (3,10) -- (3,5);
\draw[line width = .1ex, gray] (4,10) -- (4,5);
\draw[line width = .1ex, gray] (5,10) -- (5,5);
\draw[line width = .1ex, gray] (6,10) -- (6,5);
\draw[line width = .1ex, gray] (7,10) -- (7,2);
\draw[line width = .1ex, gray] (8,8) -- (8,2);
\draw[line width = .1ex, gray] (9,8) -- (9,0);

\draw[line width = .25ex] (0,6)--(0,10)--(8,10)--(8,8)--(10,8)--(10,0)--(8,0)--(8,2)--(6,2)--(6,5)--(2,5)--(2,6)--(0,6);
\filldraw[black] (0,9.5) circle (1ex);
\filldraw[black] (0,8.5) circle (1ex);
\filldraw[black] (2,5.5) circle (1ex);
\filldraw[black] (6,4.5) circle (1ex);

\filldraw[black] (5.5,5) circle (1ex);
\filldraw[black] (6.5,2) circle (1ex);
\filldraw[black] (7.5,2) circle (1ex);
\filldraw[black] (9.5,0) circle (1ex);
\put(118,-13){{\small{$H_1$}}}
\put(87,14){{\small{$H_2$}}}
\put(73,14){{\small{$H_3$}}}
\put(52,55){{\small{$H_4$}}}
\put(64,45){{\small{$V_3$}}}
\put(12,67){{\small{$V_4$}}}
\put(-15,105){{\small{$V_2$}}}
\put(-15,122){{\small{$V_1$}}}
\end{tikzpicture}}
\end{picture}
\]    
\end{example}

A \mydef{lattice path} from $H_i$ to $V_i$ in the region $\lambda$ is a path from $H_i$ to $V_i$ in $\lambda$ consisting of up and left steps in  $\lambda$. We visualize lattice paths with tiles
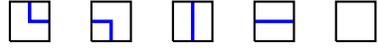
\[\begin{tikzpicture}[scale=0.25]
\begin{scope}[scale=0.7,thick]
\draw[line width = .25ex, blue] (1.5,7) -- (1.5,5.5)--(3,5.5);
\draw (0,4) -- (0,7)--(3,7)--(3,4)--(0,4);
\end{scope}
\begin{scope}[scale=0.7,xshift=15em,thick]
\draw[line width = .25ex, blue] (1.5,4) -- (1.5,5.5)--(0,5.5);
\draw (0,4) -- (0,7)--(3,7)--(3,4)--(0,4);
\end{scope}
\begin{scope}[scale=0.7,xshift=30em,thick]
\draw[line width = .25ex, blue] (1.5,4) -- (1.5,7);
\draw (0,4) -- (0,7)--(3,7)--(3,4)--(0,4);
\end{scope}
\begin{scope}[scale=0.7,xshift=45em,thick]
\draw[line width = .25ex, blue] (0,5.5) -- (3,5.5);
\draw (0,4) -- (0,7)--(3,7)--(3,4)--(0,4);
\end{scope}
\begin{scope}[scale=0.7,xshift=60em,thick]
\draw 
(0,4) -- (0,7)--(3,7)--(3,4)--(0,4);
\end{scope}
\end{tikzpicture}.
\]

 We call the leftmost tile an \mydef{SE-elbow} tile and the rightmost tile a \mydef{blank} tile. For a box ${\sf b}\in L$, let $t({\sf b})$ denote the tile occupying ${\sf b}$.

For a minimal two-sided ladder $(L,{\mathcal M})$ where $L=\lambda/\mu$, define $\nilp{(L,{\mathcal{M}})}$ to be the set of non-intersecting lattice paths $P=(P_1,\ldots,P_{N})$ in which $P_{i}$ is a lattice path in $\lambda$ from $H_i$ to $V_i$, for $H_i\in H({\mathcal M}),V_i\in V({\mathcal M})$, where a path $P_i$ may occupy box $(x,y)\in\mu$ only if:
\begin{enumerate}
    \item $t(x,y)\neq \SEelb$, and 
    \item $t(x+k,y-k)\neq\blank$ \,for any $(x+k,y-k)\in\lambda$ where $k\in\mathbb{Z}_{>0}$.
\end{enumerate}

\begin{example} For $(L,\mathcal{M})$ as in Example~\ref{ex:laddAlgIll}, the leftmost two diagrams are in $\nilp{(L,{\mathcal{M}})}$. The rightmost diagram is not since $P_1$ occupies $(2,9)\in\mu$ but $t(2+1,9-1)=t(3,8)=\blank$.
    \[\begin{picture}(340,80)
\put(0,-15){\begin{tikzpicture}[scale=.3]
\draw[line width = .1ex, gray] (0,9) -- (8,9);
\draw[line width = .1ex, gray] (0,8) -- (8,8);
\draw[line width = .1ex, gray] (0,7) -- (10,7);
\draw[line width = .1ex, gray] (2,6) -- (10,6);
\draw[line width = .1ex, gray] (6,5) -- (10,5);
\draw[line width = .1ex, gray] (6,4) -- (10,4);
\draw[line width = .1ex, gray] (6,3) -- (10,3);
\draw[line width = .1ex, gray] (6,2) -- (10,2);
\draw[line width = .1ex, gray] (8,1) -- (10,1);

\draw[line width = .1ex, gray] (1,10) -- (1,6);
\draw[line width = .1ex, gray] (2,10) -- (2,6);
\draw[line width = .1ex, gray] (3,10) -- (3,5);
\draw[line width = .1ex, gray] (4,10) -- (4,5);
\draw[line width = .1ex, gray] (5,10) -- (5,5);
\draw[line width = .1ex, gray] (6,10) -- (6,5);
\draw[line width = .1ex, gray] (7,10) -- (7,2);
\draw[line width = .1ex, gray] (8,8) -- (8,2);
\draw[line width = .1ex, gray] (9,8) -- (9,0);

\draw[line width = .25ex] (0,6)--(0,10)--(8,10)--(8,8)--(10,8)--(10,0)--(8,0)--(8,2)--(6,2)--(6,5)--(2,5)--(2,6)--(0,6);

\draw[line width = .25ex, blue] (5.5,5)--(5.5,5.5)--(2,5.5);
\draw[line width = .25ex, blue] (6.5,2)--(6.5,4.5)--(6,4.5);
\draw[line width = .25ex, blue] (7.5,2)--(7.5,5.5)--(6.5,5.5)--(6.5,6.5)--(0.5,6.5)--(0.5,8.5)--(0,8.5);
\draw[line width = .25ex, blue] (9.5,0)--(9.5,0.5)--(8.5,0.5)--(8.5,6.5)--(7.5,6.5)--(7.5,7.5)--(1.5,7.5)--(1.5,9.5)--(0,9.5);

\filldraw[black] (0,9.5) circle (1ex);
\filldraw[black] (0,8.5) circle (1ex);
\filldraw[black] (2,5.5) circle (1ex);
\filldraw[black] (6,4.5) circle (1ex);

\filldraw[black] (5.5,5) circle (1ex);
\filldraw[black] (6.5,2) circle (1ex);
\filldraw[black] (7.5,2) circle (1ex);
\filldraw[black] (9.5,0) circle (1ex);

\end{tikzpicture}}
\put(130,-15){\begin{tikzpicture}[scale=.3]
\draw[line width = .1ex, gray] (0,9) -- (8,9);
\draw[line width = .1ex, gray] (0,8) -- (8,8);
\draw[line width = .1ex, gray] (0,7) -- (10,7);
\draw[line width = .1ex, gray] (2,6) -- (10,6);
\draw[line width = .1ex, gray] (6,5) -- (10,5);
\draw[line width = .1ex, gray] (6,4) -- (10,4);
\draw[line width = .1ex, gray] (6,3) -- (10,3);
\draw[line width = .1ex, gray] (6,2) -- (10,2);
\draw[line width = .1ex, gray] (8,1) -- (10,1);

\draw[line width = .1ex, gray] (1,10) -- (1,6);
\draw[line width = .1ex, gray] (2,10) -- (2,6);
\draw[line width = .1ex, gray] (3,10) -- (3,5);
\draw[line width = .1ex, gray] (4,10) -- (4,5);
\draw[line width = .1ex, gray] (5,10) -- (5,5);
\draw[line width = .1ex, gray] (6,10) -- (6,5);
\draw[line width = .1ex, gray] (7,10) -- (7,2);
\draw[line width = .1ex, gray] (8,8) -- (8,2);
\draw[line width = .1ex, gray] (9,8) -- (9,0);

\draw[line width = .25ex] (0,6)--(0,10)--(8,10)--(8,8)--(10,8)--(10,0)--(8,0)--(8,2)--(6,2)--(6,5)--(2,5)--(2,6)--(0,6);

\draw[line width = .25ex, blue] (5.5,5)--(5.5,5.5)--(2,5.5);
\draw[line width = .25ex, blue] (6.5,2)--(6.5,4.5)--(6,4.5);
\draw[line width = .25ex, blue] (7.5,2)--(7.5,5.5)--(6.5,5.5)--(6.5,7.5)--(0.5,7.5)--(0.5,8.5)--(0,8.5);
\draw[line width = .25ex, blue] (9.5,0)--(9.5,7.5)--(9.5,7.5)--(7.5,7.5)--(7.5,9.5)--(0,9.5);

\filldraw[black] (0,9.5) circle (1ex);
\filldraw[black] (0,8.5) circle (1ex);
\filldraw[black] (2,5.5) circle (1ex);
\filldraw[black] (6,4.5) circle (1ex);

\filldraw[black] (5.5,5) circle (1ex);
\filldraw[black] (6.5,2) circle (1ex);
\filldraw[black] (7.5,2) circle (1ex);
\filldraw[black] (9.5,0) circle (1ex);

\end{tikzpicture}}
\put(260,-15){\begin{tikzpicture}[scale=.3]
\draw[line width = .1ex, gray] (0,9) -- (8,9);
\draw[line width = .1ex, gray] (0,8) -- (8,8);
\draw[line width = .1ex, gray] (0,7) -- (10,7);
\draw[line width = .1ex, gray] (2,6) -- (10,6);
\draw[line width = .1ex, gray] (6,5) -- (10,5);
\draw[line width = .1ex, gray] (6,4) -- (10,4);
\draw[line width = .1ex, gray] (6,3) -- (10,3);
\draw[line width = .1ex, gray] (6,2) -- (10,2);
\draw[line width = .1ex, gray] (8,1) -- (10,1);

\draw[line width = .1ex, gray] (1,10) -- (1,6);
\draw[line width = .1ex, gray] (2,10) -- (2,6);
\draw[line width = .1ex, gray] (3,10) -- (3,5);
\draw[line width = .1ex, gray] (4,10) -- (4,5);
\draw[line width = .1ex, gray] (5,10) -- (5,5);
\draw[line width = .1ex, gray] (6,10) -- (6,5);
\draw[line width = .1ex, gray] (7,10) -- (7,2);
\draw[line width = .1ex, gray] (8,8) -- (8,2);
\draw[line width = .1ex, gray] (9,8) -- (9,0);

\draw[line width = .25ex] (0,6)--(0,10)--(8,10)--(8,8)--(10,8)--(10,0)--(8,0)--(8,2)--(6,2)--(6,5)--(2,5)--(2,6)--(0,6);

\draw[line width = .25ex, blue] (5.5,5)--(5.5,5.5)--(2,5.5);
\draw[line width = .25ex, blue] (6.5,2)--(6.5,4.5)--(6,4.5);
\draw[line width = .25ex, blue] (7.5,2)--(7.5,5.5)--(6.5,5.5)--(6.5,6.5)--(0.5,6.5)--(0.5,8.5)--(0,8.5);
\draw[line width = .25ex, blue] (9.5,0)--(9.5,8.5)--(9.5,8.5)--(7.5,8.5)--(7.5,9.5)--(0,9.5);

\filldraw[black] (0,9.5) circle (1ex);
\filldraw[black] (0,8.5) circle (1ex);
\filldraw[black] (2,5.5) circle (1ex);
\filldraw[black] (6,4.5) circle (1ex);

\filldraw[black] (5.5,5) circle (1ex);
\filldraw[black] (6.5,2) circle (1ex);
\filldraw[black] (7.5,2) circle (1ex);
\filldraw[black] (9.5,0) circle (1ex);

\end{tikzpicture}}
\end{picture}
\]    
\end{example}
 
 For $P\in\nilp{(L,{\mathcal{M}})}$, define 
\[\blk(P):=\Big\{ (i,j)\in L \, : \, t(i,j)= \blank\Big\}.\]
Let ${\wt}(L,{\mathcal{M}})=|L|-|\blk(P)|$, where $|L|$ denotes the number of boxes in $L$ and we take $P\in \nilp{(L,{\mathcal{M}})}$. By the definition of $\nilp{(L,{\mathcal{M}})}$, $|\blk(P)|$ is constant across all paths $P\in \nilp{(L,{\mathcal{M}})}$. 
For $P\in \nilp{(L,{\mathcal{M}})}$, define the \mydef{unforced elbows} of $P$ as the set
\[\elb(P):=\{ (i,j)\in L \, : \, t(i,j)= \SEelb \, \text{ and } (i-k,j+k)\in \blk(P) \text{ for some } k>0\}.\] 
Define the map 
\begin{align*}
       \psi: \nilp{(L,{\mathcal{M}})}&\rightarrow \{S\subseteq[n]^2\}\\
       P&\mapsto \blk(P).
    \end{align*}
For a minimal two-sided ladder $(L,{\mathcal M})$, take $P_{{\tt bot}}(L,{\mathcal M})\in \nilp{(L,{\mathcal{M}})}$ to be the non-intersecting lattice path in which each path lies maximally to the lower left in $L$. 

Ordinary excited Young diagrams naturally biject with non-intersecting lattice paths when $L=\lambda$, as shown by Kreiman \cite{KreimanEYD}. We prove the corresponding  bijection for skew excited Young diagrams.

\begin{proposition}\label{prop:laddToSeyd}
Let $(L,{\mathcal M})$ be a minimal two-sided ladder. 
    Then for $(v,w)=\perm(L,{\mathcal{M}})$,
    \begin{align*}
       \psi: \nilp{(L,{\mathcal{M}})}&\rightarrow \skewexcited(v,w)\\
       P&\mapsto \blk(P)
    \end{align*}
    is a bijective map, where $\psi(P_{{\tt bot}}(L,{\mathcal M}))=\maxexcited(v,w)$.
\end{proposition}
\begin{proof}
We first show $\psi(P_{{\tt bot}}(L,{\mathcal M}))=\maxexcited(v,w)$.
Suppose $L=\lambda/\mu$. Without loss of generality, assume that consecutive connected components in $\lambda/\mu$ (say, ordering left to right) share corners.

We proceed by induction on $|\mu|$. 
When $|\mu|=0$, $L$ is a one-sided ladder. This implies $v$ is Grassmannian, as proven in \cite[Theorem~4.7.3]{GM}. This case is proven in \cite[Section~5]{KreimanEYD}. See \cite{MPP} and \cite[Section~7.3]{Weigandt.BPD} for additional discussion.

Suppose the result holds for $|\mu|=k-1\geq0$. Consider $L=\lambda/\mu$ such that $|\mu|=k$. Take $(a,b)\in L$ to be the topmost box in $L$ such that $(a+1,b+1)\in L$ and $(a,b+1)\not\in L$. 
Let $L':=L\cup\{(a,b+1)\}$.

Note that $(L',{\mathcal{M}})$ satisfies all minimality conditions, so we may apply the  inductive assumption to $(L',{\mathcal{M}})$.
Let $(v',w')=\perm(L',\mathcal{M})$. 
 Since $L'=L\cup\{(a,b+1)\}$, this determines 
\begin{equation}\label{eq:LswChange}
    (L')^{\sf{NE}}=L^{\sf{NE}}-\{(a-1,b),(a,b+1)\}\cup \{(a-1,b+1)\}. 
\end{equation}
Then by definition of $v$ and choice of $(a,b)$, $v=s_iv'$, where $i$ is the label of $\phi_{v'}^{-1}(a,b+1)$ when constructing ${\sf word}(D(v'))$.

Suppose $t(a,b+1)=\blank$ in $P_{{\tt bot}}(L',{\mathcal M})$. 
Using the inductive assumption, we know that $(a,b+1)\in\maxexcited(v',w')$. By the definition of $\perm(L,\mathcal{M})$ and Equation~\eqref{eq:LswChange}, it is straightforward to check through case analysis on $L^{\sf{NE}}\cap\{(a-1,b),(a,b+1)\}$ that $w=s_iw'$. 
Thus $\maxexcited(v,w)=\maxexcited(v',w')-\{(a,b+1)\}$.
Since 
\[\blk(P_{{\tt bot}}(L,{\mathcal M}))=\blk(P_{{\tt bot}}(L',{\mathcal M}))-\{(a,b+1)\},\] 
we find $\psi(P_{{\tt bot}}(L,{\mathcal M}))=\maxexcited(v,w)$.

Otherwise $t(a,b+1)\neq \blank$ in $P_{{\tt bot}}(L',{\mathcal M})$. 
By the inductive assumption, we know that $(a,b+1)\not\in\maxexcited(v',w')$. By the definition of $\perm(L,\mathcal{M})$ and Equation~\eqref{eq:LswChange}, it is straightforward to check through case analysis on $L^{\sf{NE}}\cap\{(a-1,b),(a,b+1)\}$ that $w=w'$. 
Thus $\maxexcited(v,w)=\maxexcited(v',w')$. Since $\blk(P_{{\tt bot}}(L,{\mathcal M}))=\blk(P_{{\tt bot}}(L',{\mathcal M}))$, we have $\psi(P_{{\tt bot}}(L,{\mathcal M}))=\maxexcited(v,w)$.
 
    We see excited moves and local moves on lattice paths biject with each other:
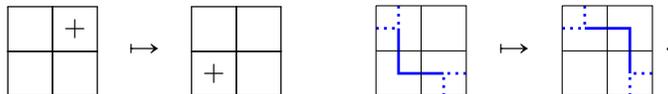
\[
    \ytableausetup{boxsize=1.2em}
\begin{ytableau}
\, & + \\
& \ 
\end{ytableau}
\hspace{1em}
 \raisebox{-.2em}{$\mapsto$}
 \hspace{1em}
\begin{ytableau}
\ &\\
+ &
\end{ytableau} 
\hspace{3em}
 \raisebox{-1.5em}{\begin{tikzpicture}[scale=0.15]
\begin{scope}
\draw[line width = .2ex, blue] (2,6) -- (2,2)--(6,2);
\draw[dotted,line width = .2ex, blue] (2,8) -- (2,6);
\draw[dotted,line width = .2ex, blue] (0,6) -- (2,6);
\draw[dotted,line width = .2ex, blue] (8,2) -- (6,2);
\draw[dotted,line width = .2ex, blue] (6,0) -- (6,2);
\draw (0,8) -- (8,8)-- (8,0)-- (0,0)-- (0,8);
\draw (4,8) -- (4,0);
\draw (8,4) -- (0,4);
\end{scope}
\end{tikzpicture}}
\hspace{1em}
\raisebox{-.2em}{$\mapsto$}
\hspace{1em}
 \raisebox{-1.5em}{\begin{tikzpicture}[scale=0.15]
\begin{scope}
\draw[line width = .2ex, blue] (2,6)--(6,6)--(6,2);
\draw[dotted,line width = .2ex, blue] (2,8) -- (2,6);
\draw[dotted,line width = .2ex, blue] (0,6) -- (2,6);
\draw[dotted,line width = .2ex, blue] (8,2) -- (6,2);
\draw[dotted,line width = .2ex, blue] (6,0) -- (6,2);
\draw (0,8) -- (8,8)-- (8,0)-- (0,0)-- (0,8);
\draw (4,8) -- (4,0);
\draw (8,4) -- (0,4);
\end{scope}
\end{tikzpicture}} \ .
\]
Therefore since $\mathcal{R}_v=\lambda/\mu$, the result follows.
\end{proof}

Then for $(v,w)=\perm(L,{\mathcal{M}})$, define $P_{{\tt zip}}(L,{\mathcal M}):=\psi^{-1}(D_{\tt zip}(v,w))$.
 
\begin{example}\label{ex:NILPtoSEYD}
Take $(L,\mathcal{M})$ as in Example~\ref{ex:laddAlgIll}. To the left are $P_{{\tt bot}}(L,{\mathcal M})$ and $\maxexcited(v,w)$. 
To the right we have $P_{{\tt zip}}(L,{\mathcal M})$ and $\Kzip(v,w)$ where plusses in $\Kzip(v,w)-\zip(v,w)$ are drawn in bold blue. Note that $\elb(P_{{\tt zip}}(L,{\mathcal M}))$ coincides with $\Kzip(v,w)-\zip(v,w)$.
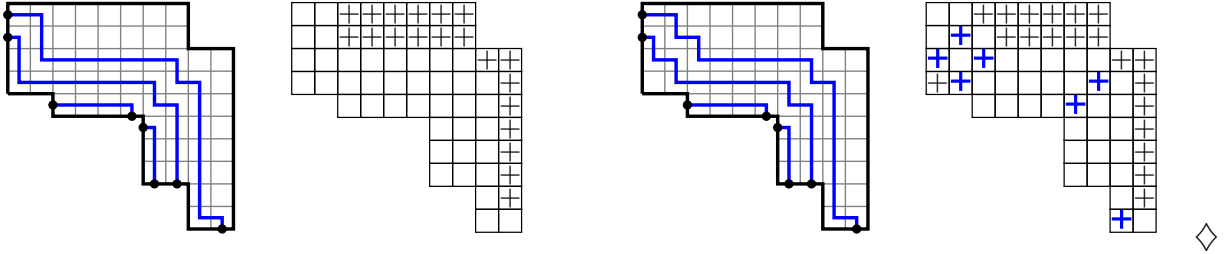
\[\begin{picture}(450,90)
\put(0,-15){\begin{tikzpicture}[scale=.35]
\draw[line width = .1ex, gray] (0,9) -- (8,9);
\draw[line width = .1ex, gray] (0,8) -- (8,8);
\draw[line width = .1ex, gray] (0,7) -- (10,7);
\draw[line width = .1ex, gray] (2,6) -- (10,6);
\draw[line width = .1ex, gray] (6,5) -- (10,5);
\draw[line width = .1ex, gray] (6,4) -- (10,4);
\draw[line width = .1ex, gray] (6,3) -- (10,3);
\draw[line width = .1ex, gray] (6,2) -- (10,2);
\draw[line width = .1ex, gray] (8,1) -- (10,1);

\draw[line width = .1ex, gray] (1,10) -- (1,6);
\draw[line width = .1ex, gray] (2,10) -- (2,6);
\draw[line width = .1ex, gray] (3,10) -- (3,5);
\draw[line width = .1ex, gray] (4,10) -- (4,5);
\draw[line width = .1ex, gray] (5,10) -- (5,5);
\draw[line width = .1ex, gray] (6,10) -- (6,5);
\draw[line width = .1ex, gray] (7,10) -- (7,2);
\draw[line width = .1ex, gray] (8,8) -- (8,2);
\draw[line width = .1ex, gray] (9,8) -- (9,0);

\draw[line width = .25ex] (0,6)--(0,10)--(8,10)--(8,8)--(10,8)--(10,0)--(8,0)--(8,2)--(6,2)--(6,5)--(2,5)--(2,6)--(0,6);

\draw[line width = .25ex, blue] (5.5,5)--(5.5,5.5)--(2,5.5);
\draw[line width = .25ex, blue] (6.5,2)--(6.5,4.5)--(6,4.5);
\draw[line width = .25ex, blue] (7.5,2)--(7.5,5.5)--(6.5,5.5)--(6.5,6.5)--(0.5,6.5)--(0.5,8.5)--(0,8.5);
\draw[line width = .25ex, blue] (9.5,0)--(9.5,0.5)--(8.5,0.5)--(8.5,6.5)--(7.5,6.5)--(7.5,7.5)--(1.5,7.5)--(1.5,9.5)--(0,9.5);

\filldraw[black] (0,9.5) circle (1ex);
\filldraw[black] (0,8.5) circle (1ex);
\filldraw[black] (2,5.5) circle (1ex);
\filldraw[black] (6,4.5) circle (1ex);

\filldraw[black] (5.5,5) circle (1ex);
\filldraw[black] (6.5,2) circle (1ex);
\filldraw[black] (7.5,2) circle (1ex);
\filldraw[black] (9.5,0) circle (1ex);

\end{tikzpicture}}
\put(110,76){\ytableausetup
{boxsize=0.8em}
{\begin{ytableau}
  \ &  \ &  + &  + & + &  + &  + &  + \\
  \ &  \ &  + &  + & + &  + &  + &  + \\
\ &  \ & \ &  \ & \ &  \ &  \ &   \ &  + &  +\\
\ &  \ &  \ &  \ & \ &  \ &  \ &  \ &   \ &  +\\
 \none &  \none &  \ &  \ & \ &  \ &   \ &   \ &   \ &  + \\
 \none &  \none &  \none &  \none &  \none &  \none &   \ &  \ &  \ &  +\\
 \none &  \none &  \none &  \none &  \none &  \none &   \ &  \ &  \ &  +\\
 \none &  \none &  \none &  \none &  \none &  \none &   \ &  \ & \ &  +\\
 \none &  \none &  \none &  \none &  \none &  \none &   \none &  \none &  \ &  +\\
 \none &  \none &  \none &  \none &  \none &  \none &   \none &  \none &  \ &  \ \\
\end{ytableau}}}
\put(235,-15){\begin{tikzpicture}[scale=.35]
\draw[line width = .1ex, gray] (0,9) -- (8,9);
\draw[line width = .1ex, gray] (0,8) -- (8,8);
\draw[line width = .1ex, gray] (0,7) -- (10,7);
\draw[line width = .1ex, gray] (2,6) -- (10,6);
\draw[line width = .1ex, gray] (6,5) -- (10,5);
\draw[line width = .1ex, gray] (6,4) -- (10,4);
\draw[line width = .1ex, gray] (6,3) -- (10,3);
\draw[line width = .1ex, gray] (6,2) -- (10,2);
\draw[line width = .1ex, gray] (8,1) -- (10,1);

\draw[line width = .1ex, gray] (1,10) -- (1,6);
\draw[line width = .1ex, gray] (2,10) -- (2,6);
\draw[line width = .1ex, gray] (3,10) -- (3,5);
\draw[line width = .1ex, gray] (4,10) -- (4,5);
\draw[line width = .1ex, gray] (5,10) -- (5,5);
\draw[line width = .1ex, gray] (6,10) -- (6,5);
\draw[line width = .1ex, gray] (7,10) -- (7,2);
\draw[line width = .1ex, gray] (8,8) -- (8,2);
\draw[line width = .1ex, gray] (9,8) -- (9,0);

\draw[line width = .25ex] (0,6)--(0,10)--(8,10)--(8,8)--(10,8)--(10,0)--(8,0)--(8,2)--(6,2)--(6,5)--(2,5)--(2,6)--(0,6);

\draw[line width = .25ex, blue] (5.5,5)--(5.5,5.5)--(2,5.5);
\draw[line width = .25ex, blue] (6.5,2)--(6.5,4.5)--(6,4.5);
\draw[line width = .25ex, blue] (7.5,2)--(7.5,5.5)--(6.5,5.5)--(6.5,6.5)--(1.5,6.5)--(1.5,7.5)--(0.5,7.5)--(0.5,8.5)--(0,8.5);
\draw[line width = .25ex, blue] (9.5,0)--(9.5,0.5)--(8.5,0.5)--(8.5,6.5)--(7.5,6.5)--(7.5,7.5)--(2.5,7.5)--(2.5,8.5)--(1.5,8.5)--(1.5,9.5)--(0,9.5);

\filldraw[black] (0,9.5) circle (1ex);
\filldraw[black] (0,8.5) circle (1ex);
\filldraw[black] (2,5.5) circle (1ex);
\filldraw[black] (6,4.5) circle (1ex);

\filldraw[black] (5.5,5) circle (1ex);
\filldraw[black] (6.5,2) circle (1ex);
\filldraw[black] (7.5,2) circle (1ex);
\filldraw[black] (9.5,0) circle (1ex);
\end{tikzpicture}}
\put(345,76){\ytableausetup
{boxsize=0.8em}
{\begin{ytableau}
  \ &  \ &  + &  + & + &  + &  + &  + \\
  \ &  \textcolor{blue}{\Plus} &  \ &  + & + &  + &  + &  + \\
\textcolor{blue}{\Plus} &  \ & \textcolor{blue}{\Plus} &  \ & \ &  \ &  \ &   \ &  + &  +\\
+ &  \textcolor{blue}{\Plus} &  \ &  \ & \ &  \ &  \ &  \textcolor{blue}{\Plus} &   \ &  +\\
 \none &  \none &  \ &  \ & \ &  \ &   \textcolor{blue}{\Plus} &   \ &   \ &  + \\
 \none &  \none &  \none &  \none &  \none &  \none &   \ &  \ &  \ &  + \\
 \none &  \none &  \none &  \none &  \none &  \none &   \ &  \ &  \ &  + \\
 \none &  \none &  \none &  \none &  \none &  \none &   \ &  \ & \ &  + \\
 \none &  \none &  \none &  \none &  \none &  \none &   \none &  \none &  \ &  + \\
 \none &  \none &  \none &  \none &  \none &  \none &   \none &  \none &  \textcolor{blue}{\Plus} &  \ \\
\end{ytableau}}}
\end{picture}\]
\end{example}

\begin{lemma}\label{lem:mdDiagSame}
Let $(L,{\mathcal M})$ be a minimal two-sided ladder. Suppose $\perm(L,\mathcal{M})=(v,w)$. Then $v,w\in \avoid{n}$ are {\UD}, and  $\gap({\sf b},{\sf c})\in\{ 0,\infty\}$ for all ${\sf b},{\sf c}\in \maxexcited{(v,w)}$.
\end{lemma}
\begin{proof}
Consider the partition $\maxexcited{(v,w)}=\sqcup_{q\in[m]}\Gamma_q$ into connected components. Note each connected component $\Gamma_q$ will be a smaller skew shape.
Define $w^{(q)}=\delta(\phi_v^{-1}(\Gamma_q))$ 
for $q\in[m]$. 
Then $w^{(q)}\in \avoid{n}$ by Proposition~\ref{prop:321-graph}.
    We claim the following gives a bijection: 
    \begin{equation}\label{eq:seydDecomp}
    \begin{aligned}
            \skewexcited(v,w^{(1)}) \times \cdots \times \skewexcited(v,w^{(m)}) &\longrightarrow \skewexcited(v,w),\\
 (D_1,\dots,D_m) &\longmapsto D_1 \sqcup \cdots \sqcup D_m.
\end{aligned}
    \end{equation}
If this fails, there exists $(x,y)\in \Gamma_q$ and $(x,y+k)\in \Gamma_{q+1}$ for some $q\in[m-1]$, $k>1$ such that $(x,y+k-1)\not\in \maxexcited{(v,w)}$. By Proposition~\ref{prop:laddToSeyd},
$\psi(P_{{\tt bot}}(L,{\mathcal M}))=\maxexcited(v,w)$, so in $P_{{\tt bot}}(L,{\mathcal M})$, $t(x,y+k-1)\neq\blank$. Since $\maxexcited{(v,w)}$ is upper rightmost, this implies that some path $P_j$ in $P_{{\tt bot}}(L,{\mathcal M})$ occupies tile $(x,y+k-1)$. This forces $P_j$ to pass above $\Gamma_q$, violating condition (2) for lattice paths, a contradiction. This proves Equation~\eqref{eq:seydDecomp}.

We obtain a similar contradiction if there exists $(x,y)\in \Gamma_q$ and $(x+k,y)\in \Gamma_{q+1}$ for some $q\in[m-1]$, $k>1$ such that $(x+k-1,y)\not\in \maxexcited{(v,w)}$. This implies condition (i) of {\UD}. Condition (ii) and the fact that $\gap({\sf b},{\sf c})\in\{ 0,\infty\}$ follow by Equation~\eqref{eq:seydDecomp}.
\end{proof}

This allows us to specialize our main results.
\begin{corollary}\label{cor:laddReg}
Suppose $(L,{\mathcal M})$ is a minimal two-sided ladder where $\perm(L,\mathcal{M})=(v,w)$. Then the following hold:
    \begin{align*}
         \reg (X(L,\mathcal{M}))&=|\Kzip(v,w)|-{\sf inv}(w)=\sum_{i=0}^\ell\sum_{{\sf d}\in \Lambda^{(i)}}\fTab{v,w}^{(i)}({\sf d})\text{, and}\\
          a(X(L,\mathcal{M}))&=|\Kzip(v,w)|-{\sf inv}(v)=\Big(\sum_{i=0}^\ell\sum_{{\sf d}\in \Lambda^{(i)}}\fTab{v,w}^{(i)}({\sf d})\Big)-{\sf inv}(v)+{\sf inv}(w).
    \end{align*}
\end{corollary}    
    \begin{proof}
This follows from Theorem~\ref{thm:321321IndReg} and Corollary~\ref{cor:321321IndAInv} combined with Theorem~\ref{prop:2sidedKL} and Lemma~\ref{lem:mdDiagSame}.
    \end{proof}

In \cite[Lemma~14]{GK15} Ghorpade--Krattenthaler give an algorithm to compute $a(X(L,\mathcal{M}))$ for two-sided ladders generated by $k\,-\,$minors with additional marked points $({\sf p}, r)$, where $r\in[k]$ and ${\sf p}$ lies on the topmost vertical edge or rightmost horizontal edge of $L$. 
This algorithm computes $a(X(L,\mathcal{M}))$ by maximizing $|\elb(P)|$ for $P\in \nilp{(L,{\mathcal{M}})}$. We extend this to all minimal two-sided ladders:

\begin{corollary}\label{thm:regAnvLaddElb}
    Let  $(L,{\mathcal M})$ be a minimal two-sided ladder. Then 
    \begin{align*}
         \reg (X(L,\mathcal{M}))&=|\elb(P_{{\tt zip}}(L,{\mathcal M}))| \text{, and}\\
          a(X(L,\mathcal{M}))&=|\elb(P_{{\tt zip}}(L,{\mathcal M}))|-{\wt}(L,\mathcal{M}).
    \end{align*}
\end{corollary}
\begin{proof}

    From Proposition~\ref{prop:laddToSeyd} and the definition of $\elb(P)$ where $P\in \nilp{(L,{\mathcal{M}})}$, it is straightforward to see
    \[\elb(P_{{\tt zip}}(L,{\mathcal M}))=\elb(\psi^{-1}(D_{\tt zip}(v,w)))=D_{\tt zip}^K(v,w)-D_{\tt zip}(v,w).\]
    Since ${\sf inv}(w)=|D_{\tt zip}(v,w)|$, the first result follows. By combining this with Corollary~\ref{cor:laddReg} and the fact that ${\wt}(L,\mathcal{M})={\sf inv}(v)-{\sf inv}(w)$, the second result follows.
\end{proof}

\begin{example}\label{ex:laddThmIll}
Let $(L,\mathcal{M})$ be as in Example~\ref{ex:NILPtoSEYD}. 
By Corollary~\ref{thm:regAnvLaddElb}, 
\begin{center}
    $\reg (X(L,\mathcal{M}))=|\elb(P_{{\tt zip}}(L,{\mathcal M}))|=7$, and\\
    $a(X(L,\mathcal{M}))=|\elb(P_{{\tt zip}}(L,{\mathcal M}))|-{\wt}(L,{\mathcal{M}})=7-(60-20)=-33$.
\end{center}
\vspace*{-0.5cm}
\end{example}

In general, the lattice path constructed for Corollary~\ref{thm:regAnvLaddElb} differs from the outputted lattice path in \cite[Lemma~14]{GK15} that maximizes unforced elbows. Applying Construction~\ref{alg:LaddPointCompute}, it is straightforward to extend \cite{GK15} to all minimal two-sided ladders.

\section*{Acknowledgements}
We would like to thank Elisa Gorla, Patricia Klein, Jenna Rajchgot, Ada Stelzer, Anna Weigandt, and Alexander Yong for helpful comments and conversations. We would also like to thank the anonymous reviewer who noted errors in the earlier version of this manuscript. No AI tools were used in the development of the theorems nor proofs of this manuscript. AI was used in final proofreading, and any resulting changes, all minor edits, were implemented directly by the author.

\bibliographystyle{plainurl}
\bibliography{2sidedLadd}

\end{document}